\documentclass[12pt,a4paper]{amsart}

\usepackage[
    margin=1in,          
    footskip=0.5in,      
    hmarginratio=1:1     
]{geometry}

\usepackage{graphicx}
\usepackage{amsmath,amssymb}
\usepackage{amsthm}
\usepackage{bbm}
\usepackage{hyperref}
\usepackage{tikz}
\usetikzlibrary{patterns}
\usepackage{enumitem}
\usepackage[capitalize,nameinlink,noabbrev,nosort]{cleveref}
\usepackage{subcaption}

\theoremstyle{plain}
\newtheorem{theorem}{Theorem}[section]
\newtheorem{proposition}[theorem]{Proposition}
\newtheorem{lemma}[theorem]{Lemma}
\newtheorem{corollary}[theorem]{Corollary}
\newtheorem{question}[theorem]{Question}

\theoremstyle{definition}
\newtheorem{definition}[theorem]{Definition}
\newtheorem{conjecture}[theorem]{Conjecture}

\theoremstyle{remark}
\newtheorem{remark}[theorem]{Remark}

\newcommand{\iden}{\mathbbm{1}}
\newcommand{\qbinom}[2]{\genfrac{[}{]}{0pt}{}{#1}{#2}_q}
\DeclareMathOperator{\PD}{PD}
\DeclareMathOperator{\pd}{pd}
\DeclareMathOperator{\PSD}{PSD}
\DeclareMathOperator{\psd}{psd}

\title[Positive definite, psd and totally positive matrices over finite fields]
{Positive definite, positive semidefinite and totally positive matrices over finite fields}
\author{Arvind Ayyer}
\address{Arvind Ayyer, Department of Mathematics, 
Indian Institute of Science, Bangalore  560012, India.}
\email{arvind@iisc.ac.in}

\author{Shubhanshu Prasad}
\address{Shubhanshu Prasad, Department of Mathematics, 
Indian Institute of Science, Bangalore  560012, India.}
\email{shubhanshup@iisc.ac.in}

\date{\today}

\begin{document}

\begin{abstract}
Motivated by the equivalent definitions of positive definite (resp. positive semidefinite) matrices over real and complex fields, we give four (resp. five) inequivalent definitions for these matrices over finite fields. Our starting point is the recent definition due to Cooper--Hanna--Whitlatch (RMJ. Math., 2024) of positive elements in finite fields.
We also use this definition to study totally positive matrices over finite fields.
For all of these cases, we give explicit enumeration formulae or give bounds. 
Most of our formulas are new, but we summarize results from the existing literature for completeness.
For positive semidefinite matrices of type 5 and totally positive matrices, we give structural formulas using the rationality of the Weil zeta function, i.e. Dwork's theorem, and conjecture a quasipolynomial-type formula.
\end{abstract}


\maketitle

\section{Introduction}

Positive definite and positive semidefinite matrices are important in many areas of mathematics. For example, in optimization, they characterize convexity through the Hessian and form the foundation of semidefinite programming~\cite{VandenbergheBoyd1996}. 
In probability and statistics, positive semidefinite matrices arise as covariance matrices of random vectors~\cite{Anderson2003}. 
In functional analysis and operator theory, they appear as positive operators on Hilbert spaces~\cite{Conway1990}. 
In graph theory, Laplacian matrices are basic examples of positive semidefinite matrices~\cite{Chung1997}. 

While these matrices have been extensively studied over $\mathbb{R}$ and $\mathbb{C}$, there have been relatively few attempts to study them over finite fields, primarily because there is no canonical notion of positive elements in finite fields.
One natural notion of a positive element in a field is one which is the square of a nonzero element; this was made explicit in \cite{CooperHannaWhitlatch2022}.
The set of such nonzero elements is precisely the subgroup of quadratic residues modulo a prime power and is a classical subject in number theory. It is intimately related to the quadratic (Dirichlet) character, an important concept in modern number theory~\cite{baker-2012}.
There are many other applications of this idea.
For example, this notion of positive elements is implicitly used
in the definition of Paley graphs, which are highly symmetric graphs~\cite{brouwer-haemers-2012}. It is used in the theory of quadratic residue codes, which are an important class of error-correcting codes~\cite[Chapter 16]{macwilliams-sloane-1977}. It also has applications in cryptography and primality testing.

In a field of characteristic two, every nonzero element is positive.
In this case, one of the earliest definitions of positive definite matrices is due to Albert \cite{Albert1938}.
His result is given in \cref{sec:posdef_3}.

One of the aims in this article is to investigate how the equivalent definitions in real and complex fields (see \cref{thm:posdef RC}) generalizes to arbitrary finite fields.
We study four such inequivalent notions of positive definiteness in \cref{sec:posdef}.

More precisely, we study these in \cref{sec:posdef_1,sec:posdef 2,sec:posdef_3,sec:posdef_4},
calling them positive definite matrices of types 1 through 4 respectively.
In each case, we enumerate the number of positive definite matrices of that type.
We note that some of these results were already in the literature.
For instance, the number of positive definite matrices of types 1 and 3 are given in \cite{Vishwakarma2025} and \cite{Macwilliams1969} respectively, which we summarize for completeness.

We give new formulas for the number of positive definite matrices of types 2 and 4.
We then study the relationship between these classes of matrices in \cref{sec:rel posdef}. 
In particular, we draw Venn diagrams showing the intersections of each of these types.

We then move on to the study of positive semidefinite (psd) matrices over arbitrary finite fields generalizing the
equivalent conditions in the real and complex case in \cref{sec:psd}; see \cref{thm:psd RC}.
We study these in \cref{sec:psd 1,sec:psd 2,sec:psd 3,sec:psd 4,sec:psd 5},
calling them psd matrices of types 1 through 5 respectively. 
For psd matrices, the number of matrices of type 3 was previously determined by MacWilliams for $q$ even, and for nonsingular matrices of type 3 when $q$ is odd~\cite{Macwilliams1969}. We give new formulas for the number of psd matrices of types 1 and 4, as well as for the number of singular matrices of type 3 when $q$ is odd. We give nontrivial upper bounds for the number of psd matrices
of types 2 and 5.
Using the rationality of the Weil zeta function over finite fields, we give a structural formula for the number of psd
matrices of type 5, and conjecture that the number has a 
quasipolynomial-type behaviour.

We again study the relationship between these different definitions of positive semidefinite matrices in \cref{sec:rel psd}.

Totally positive matrices and their cousins, totally nonnegative matrices, form another important class of matrices over the real field and have a long history beginning with the work of Polya~\cite[Chapter 18]{marshall-olkin-1979}. 
They appear in a variety of areas such as systems of differential equations, approximation theory, probability, combinatorics, representation theory and algebraic geometry.

Since now we have a notion of positivity over finite fields, it is natural to study totally positive matrices in this setting; see \cref{def:tp}. 

In \cref{sec:TP}, we study some properties of the number of totally positive matrices over finite fields. We 
obtain exact formulas for the number of totally positive matrices for small order and prove a structural formula for this number again using the rationality of the Weil zeta function over finite fields.
Based on numerical experiments, we conjecture that this number has a quasipolynomial-type behaviour.

We begin with the preliminaries in \cref{sec:prelim}.

\section{Preliminaries}
\label{sec:prelim}

\subsection{Notation}

Throughout, for $m,n\in\mathbb{N}$ with $m<n$, we use $[m,n]$ to denote the set of integers from $m$ to $n$, and $[n]$ to denote the set $\{1,2,\dots,n\}$. 
The group of all the permutation of $[n]$ is denoted by $S_n$. 
We will use $p$ to denote a prime number and $q$ to denote a prime power unless stated otherwise. 
The set of all the matrices of order $m\times n$ over a field $\mathbb{F}$ is denoted by $M_{m,n}(\mathbb{F})$. When $m=n$, we simply denote this set by $M_n(\mathbb{F})$. The set of all the invertible matrices in $M_n(\mathbb{F})$ is denoted by $\mathrm{GL}_n(\mathbb{F})$. For $A\in M_n(\mathbb{F})$, $\det(A)$ denotes the determinant of $A$. For $A\in M_{m,n}(\mathbb{F})$, $A^T$ denotes the transpose of $A$. For $A\in M_{m,n}(\mathbb{F})$, $I\subseteq [m]$, and $J\subseteq [n]$, let $A_{I,J}$ denote the matrix obtained by taking the rows corresponding to $I$ and the columns corresponding to $J$. When $I=J$, we simply write $A_I$ instead of $A_{I,I}$. When $v\in \mathbb{F}_q^n$ and $I\subseteq [n]$, we use $v_I$ to denote the subvector of $v$ corresponding to the indices in $I$. The identity matrix in $M_n(\mathbb{F})$ is denoted by $\iden_n$ 
and the zero matrix in $M_{m,n}(\mathbb{F})$ is denoted by $\mathbf{0}_{m \times n}$.

\subsection{Positive definite and semidefinite matrices}

\begin{definition}
\label{def:pd psd}
For a positive integer $n$, a Hermitian matrix of order $n\times n$ over $\mathbb{C}$ is said to be \textit{positive definite} (respectively, \textit{positive semidefinite}) if for every nonzero vector $x \in \mathbb{C}^n$, $x^*Ax>0$ (respectively $x^*A x\geq0$).
\end{definition}

Over the complex field, Hermitian matrices admit several equivalent characterizations of positive definiteness, summarized in the following theorem.

\begin{theorem}[{\cite[Chapter 7]{HornJohnson2012}}]
\label{thm:posdef RC}
Let $ A \in \mathrm{GL}_n(\mathbb{C}) $ be a Hermitian matrix. The following statements are equivalent:
\begin{itemize}
\item[(a)] $A$ is positive definite.

\item[(b)] All the leading principal minors of $A$ are positive, that is, for every $m\in [n]$, $\det (A_{[m]})$ is positive.

\item[(c)] There exists a unique lower triangular matrix $L$ with positive diagonal entries such that $A = LL^*$.

\item[(d)] $A$ is a Gram matrix; that is, there exists a matrix $B \in M_{n \times k}(\mathbb{C})$ such that $A = B^*B$.

\item[(e)] $A$ is unitarily diagonalizable with strictly positive diagonal entries.

\end{itemize}
\end{theorem}

Recall that the factorization in \cref{thm:posdef RC}(c) is called the
\emph{Cholesky decomposition}. Similar to the characterization in \cref{thm:posdef RC} for positive 
definite matrices, one has the following result for positive semidefinite matrices.

\begin{theorem}
\label{thm:psd RC}
Let $A \in M_n(\mathbb{C})$ be a Hermitian matrix. The following statements are equivalent:
\begin{itemize}
\item[(a)] $A$ is positive semidefinite or psd.

\item[(b)] There exists a unique lower triangular matrix $L\in M_n(\mathbb{C})$ with nonnegative diagonal entries such that $A = LL^*$.

\item[(c)] $A$ is a Gram matrix; that is, there exists a matrix $B \in M_{n \times k}(\mathbb{C})$ such that $A = B^*B$.

\item[(d)] $A$ is unitarily diagonalizable with nonnegative diagonal entries.

\item[(e)] All principal minors of $A$ are nonnegative.

\item[(f)] There exists a subset $K\subseteq[n]$ such that $A_K$ is positive definite, and for every subset $M$ satisfying $K\subseteq M \subseteq[n]$ and $|M|\in\{r+1,r+2\}$, the principal submatrix $A_M$ is singular. 
\end{itemize}
\end{theorem}

For the equivalence of \cref{thm:psd RC}(a)-(e), the reader may again refer to \cite[Chapter 7] {HornJohnson2012}. 
The equivalence of \cref{thm:psd RC}(a) and (f) is proved in \cref{sec:new psd}. 
We now give yet another equivalent characterization of psd matrices.

\begin{theorem}
\label{thm:psd princ submatrix}
Let $A\in M_n(\mathbb{C})$ be a symmetric matrix. Then $A$ is positive semidefinite if and only if there exists a subset $K\subseteq[n]$ such that $A_K$ is positive definite and every principal submatrix $A_M$ is singular for all subsets $M\subseteq[n]$ satisfying $|M|>|K|$.
\end{theorem}

The following result is a consequence of \cref{thm:psd princ submatrix}, which is also proved in \cref{sec:new psd}.

\begin{corollary}\label{cor:posdef type 1 defn real case 1}
A matrix $A \in M_n(\mathbb{C})$ is positive semidefinite of rank $r$ if and only if, for each $m \in [n]$, there exists a subset $K_m \subseteq [m]$ such that the principal submatrix $A_{K_m}$ is positive definite, and every principal submatrix $A_M$ with $M \subseteq [m]$ and $|M|>|K_m|$ is singular.
\end{corollary}

\begin{proof}
Since every principal submatrix of a positive semidefinite matrix is positive semidefinite, the result follows by applying \cref{thm:psd princ submatrix} to each principal submatrix $A_{[m]}$.
\end{proof}

\subsection{Finite fields} 
\label{sec:finite}

In this subsection, we recall some basic properties of finite fields. For every prime $p$ and positive integer $k$, there exists a finite field with $q=p^k$ elements, unique up to isomorphism, which we denote by $\mathbb{F}_q$. For a field $\mathbb{F}$, the multiplicative group of all nonzero elements is denoted by $\mathbb{F}^{\times}$.

We now define positivity in a finite field.

\begin{definition}[\cite{CooperHannaWhitlatch2022}]
An element $a\in \mathbb{F}_q^{\times}$ is said to be \textit{positive} if there exists $b\in \mathbb{F}_q^{\times}$ such that $a=b^2$. An element in $\mathbb{F}_q^{\times}$ which is not positive is called \textit{negative}. The sets of positive and negative elements in $\mathbb{F}_q$ are denoted by $\mathbb{F}_q^{+}$ and $\mathbb{F}_q^{-}$, respectively.
\end{definition}

\begin{proposition}[\cite{CooperHannaWhitlatch2022}]
Over a finite field of characteristic two, all nonzero elements are positive. Over a finite field of odd characteristic, exactly half of the nonzero elements are positive. Therefore,
\[
|\mathbb{F}_q^{+}|=
\begin{cases}
\displaystyle q-1 & \text{if } q \text{ is even},\\[1em]
\displaystyle\frac{q-1}{2} & \text{if } q \text{ is odd}.
\end{cases}
\]
\end{proposition}

\begin{definition}
A field $\mathbb{F}$ is called \textit{definite} if for every positive element $a\in \mathbb{F}$, there exists a positive element $b\in \mathbb{F}$ such that $a=b^2$. A field which is not definite is called \textit{indefinite}.
\end{definition}

We have the following classification of definite finite fields.

\begin{proposition}[\cite{CooperHannaWhitlatch2022}]
Let $q$ be a prime power. The finite field $\mathbb{F}_q$ is definite if and only if $q\equiv 3\pmod{4}$ or $q$ is even.
\end{proposition}

An application of the pigeonhole principle gives the following proposition.

\begin{proposition}\label{prop:php fin fld}
Let $\mathbb{F}_q$ be a finite field of odd order. Then:
\begin{enumerate}
    \item Every element of $\mathbb{F}_q$ can be written as a sum of two squares in $\mathbb{F}_q$.
    \item Every positive element of $\mathbb{F}_q$ can be written as a sum of two negative elements in $\mathbb{F}_q$.
\end{enumerate}
\end{proposition}

\begin{proof}
Let $S=\mathbb{F}_q^+\cup\{0\}$. For any $a\in\mathbb{F}_q$, consider the two subsets $S$ and $a-S=\{a-s\mid s\in S\}$. Both sets have cardinality $(q+1)/2$. Since
$|S|+|a-S|=q+1>q=|\mathbb{F}_q|$,
the pigeonhole principle implies that $S$ and $a-S$ have a nonempty intersection. Hence, there exist $s_1,s_2\in S$ such that
$s_1=a-s_2$. Therefore, $a=s_1+s_2,$ which proves the first part.

For the second statement, apply a similar argument to the sets $\mathbb{F}_q^{-}$ 
and $a-\mathbb{F}_q^{-}=\{a-n\mid n\in \mathbb{F}_q^{-}\}$, where $a \in \mathbb{F}_q^+$. Note that neither set contains either $0$ or $a$.
Then the pigeonhole principle implies that these two sets have a nonempty intersection. 
Therefore, there exist $n_1,n_2\in \mathbb{F}_q^{-}$ such that $n_1=a-n_2$. Consequently, $a=n_1+n_2$, which proves the second part.
\end{proof}

We now give formulas for the number of solutions of certain equations over finite fields.

\begin{theorem}[\cite{Aabrandt2018}]\label{thm:x^2+y^2=1}
    The number of solutions to the equation $x^2+y^2=1$ over $\mathbb{F}_q$ is 
\[
    \begin{cases} 
    \displaystyle
    q & \text{if }q\text{ is even}, \\

    \displaystyle
    q+1 & \text{if }q \equiv 3 \pmod{4}, \\

    \displaystyle
    q-1 & \text{if }q \equiv 1 \pmod{4}.
    \end{cases}
    \]
\end{theorem}

The following result characterizes the existence of square roots of $-1$ in finite fields, the proof of which is elementary and omitted.

\begin{theorem}\label{thm:x^2=-1}
    The number of solutions to the equation $x^2=-1$ over $\mathbb{F}_q$ is
\[
    \begin{cases} 
    \displaystyle
    1
    & \text{if }q\text{ is even}, \\

    \displaystyle
    0 
    & \text{if }q \equiv 3 \pmod{4}, \\

    \displaystyle
    2
    & \text{if }q \equiv 1 \pmod{4}.
    \end{cases}
\]
\end{theorem}

\begin{theorem}
\label{thm:a a-1 is +ve}
    The number of elements $y\in \mathbb{F}_q^{+}$ such that $y-1\in \mathbb{F}_q^{+}$ is
\[
    \begin{cases}
    \displaystyle
    q-2
    & \text{if $q$ is even}, \\[1em]

    \displaystyle
    \frac{q-5}{4}
    & \text{if $q\equiv 1 \pmod 4$}, \\[1em]

    \displaystyle
    \frac{q-3}{4}
    & \text{if $q\equiv 3 \pmod 4$}.
    \end{cases}
\]
\end{theorem}

\begin{proof}
Observe that $y$ cannot take the values $0$ or $1$. If $q$ is even, then every nonzero element of $\mathbb{F}_q$ is positive. Hence, $y$ can be any element of $\mathbb{F}_q \setminus \{0,1\}$, giving a total of $q-2$ valid choices.

Now suppose that $q$ is odd. We need to determine the number of elements $y = a^2$ such that $a^2 - 1 = b^2$ for some $a, b \in \mathbb{F}_q^\times$ which is equivalent to saying $(a-b)(a+b)=1$. Letting $x=a+b$, we get 
\[
a=\frac{x+x^{-1}}{2}, \quad b= \frac{x-x^{-1}}{2}.
\]
To avoid the values $a^2 = 0$ or $1$, $x$ cannot be $0, 1, -1$, or satisfy $x^2 = -1$. Therefore, by \cref{thm:x^2=-1}, we must restrict to $x \in \mathbb{F}_q \setminus S$, where
\[
S=\begin{cases}
{0, 1, -1} & \text{if } q \equiv 3 \pmod{4}, \\
{0, 1, -1, x_0, -x_0} & \text{if } q \equiv 1 \pmod{4},\ \text{with } x_0^2 = -1.
\end{cases}
\]

Now, many values from the set $\mathbb{F}_q \setminus S$ may give the same value of $y=a^2$. To address this, we define a relation $\sim$ on $\mathbb{F}_q \setminus S$ by $x_1 \sim x_2$ if
\[
\left( \frac{x_1 + x_1^{-1}}{2} \right)^2 = \left( \frac{x_2 + x_2^{-1}}{2} \right)^2.
\]
It is easy to see that $\sim$ is an equivalence relation.
We now claim that each equivalence class consists of exactly four elements of the form 
\[
\{x, -x, x^{-1}, -x^{-1}\}.
\]
Observe that
\[
\left( \frac{x_1 + x_1^{-1}}{2} \right)^2 = \left( \frac{x_2 + x_2^{-1}}{2} \right)^2
\]
if and only if
\[
\frac{x_1 + x_1^{-1}}{2} = \frac{x_2 + x_2^{-1}}{2}
\quad \text{or} \quad
\frac{x_1 + x_1^{-1}}{2} = -\left( \frac{x_2 + x_2^{-1}}{2} \right).
\]
Simplifying the first case gives $x_2 = x_1$ or $x_2 = x_1^{-1}$, while the second case gives $x_2 = -x_1$ or $x_2 = -x_1^{-1}$. Moreover, for $x \in \mathbb{F}_q \setminus S$, the elements $x$, $-x$, $x^{-1}$, and $-x^{-1}$ are all distinct and lie in $\mathbb{F}_q \setminus S$. This proves the claim.

Therefore, the number of distinct values of $y$ equals the number of equivalence classes, namely
\[ 
\begin{cases} 
\displaystyle \frac{q - 5}{4} & \text{if } q \equiv 1 \pmod{4}, \\[1em]
\displaystyle \frac{q - 3}{4} & \text{if } q \equiv 3 \pmod{4}, \end{cases}
\]
completing the proof.  
\end{proof}

The number of common neighbors of two neighboring vertices in a Paley graph can also be determined using \cref{thm:a a-1 is +ve}.

\begin{corollary}
\label{cor:+-}
    Let $\mathbb{F}_q$ be a field of odd order. The number of elements $y\in \mathbb{F}_q^{+}$ such that $y+1\in \mathbb{F}_q^{-}$ is
\[
    \begin{cases}
    \displaystyle
    \frac{q-1}{4}
    & \text{if $q\equiv 1 \pmod 4$}, \\[1em]
    \displaystyle
    \frac{q+1}{4}
    & \text{if $q\equiv 3 \pmod 4$}.
    \end{cases}
\]
\end{corollary}

\begin{proof}
For $q\equiv 1\pmod{4}$, the set of positive elements can be partitioned into three subsets: those positive elements whose sum with $1$ is positive, those whose sum with $1$ is negative, and the unique positive element whose sum with $1$ is $0$. 
By \cref{thm:a a-1 is +ve}, the cardinality of the first subset is known, while the third subset has cardinality $1$. Since the total number of positive elements is $(q-1)/2$, the cardinality of the second subset is obtained by subtraction, yielding the desired result.

The case $q\equiv 3\pmod{4}$ follows by the same argument, except that the third subset is empty.
\end{proof}

\subsection{Symmetric matrices over finite fields} 

The following lemma will prove useful to us. 

\begin{lemma}[\cite{Macwilliams1969}]
\label{lem:mac symm class}
Let $\mathbb{F}$ be an arbitrary field, $A\in M_{n-1}(\mathbb{F})$ be a symmetric matrix of rank $r$, and 
\[
B=
\left(
\begin{array}{c|c}
A&v\\
\hline
\\[-1em]
v^T&a_2
\end{array}
\right),
\]
be a symmetric matrix in $M_n(\mathbb{F})$, where $v\in \mathbb{F}^{n-1}$ and $a_2\in\mathbb{F}$. 
Then:
    \begin{enumerate}
        \item The rank of $B$ is $r$ if and only if the last row of $B$ linearly depends on the first $n-1$ rows of $B$.
        \item The rank of $B$ is $r+1$ if and only if $v$ linearly depends on the first $n-1$ rows of $A$ but the last row of $B$ is linearly independent of the first $n-1$ rows of $B$.
        \item The rank of $B$ is $r+2$ if and only if $v_1$ is linearly independent of the first $n-1$ rows of $A$.
    \end{enumerate}
\end{lemma}

We now give a recursion for the number of symmetric matrices of rank $r$ in $M_n(\mathbb{F}_q)$.

\begin{theorem}[\cite{Macwilliams1969}]\label{thm:symm rank conunt}
For a positive integer $n$ and nonnegative integer $r$, let $N(n,r)$ denote the number of rank $r$ symmetric matrices in $M_n(\mathbb{F}_q)$. Then $N(n,r)$ satisfies the recurrence
\[
N(n,r)=q^rN(n-1,r)+(q^{r}-q^{r-1})N(n-1,r-1)+(q^{n}-q^{r-1})N(n-1,r-2).
\]
where we set $N(n,0)=1$ whenever $n\geq 0$, and $N(n,m)=0$ whenever $m>n$ or $m<0$. As a consequence, 
\[
N(n,r)=\prod_{i=1}^{\lfloor r/2\rfloor}\frac{q^{2i}}{q^{2i}-1}\prod_{i=0}^{r-1}(q^{n-i}-1).
\]
\end{theorem}

\subsection{Finitely generated modules over a PID}\label{sec:fg_mod_PID}
\begin{definition}
    A matrix $A \in \mathrm{GL}_n(\mathbb{F})$ is called \textit{orthogonal} if $A^{T} A = \iden_n$. The set of all orthogonal matrices over the finite field $\mathbb{F}_q$ forms a subgroup of $\mathrm{GL}_n(\mathbb{F}_q)$. We denote this subgroup by $\mathrm{O}_n(\mathbb{F}_q)$.
\end{definition}

We first recall some results from the structure theory of finitely generated modules over a principal ideal domain (PID). For the proof of the theorem, reader may  refer to Chapter 12 of \cite{DummitFoote1991}.

\begin{theorem}[{Fundamental theorem of modules over a PID~ \cite[Chapter 12]{DummitFoote1991}}]
\label{thm:fund thm mod PID}
Let $R$ be a PID, and let $\mathcal{M}$ be a finitely generated $R$-module. 
Then there exists primes $p_1, p_2, \dots, p_m \in R$ (not necessarily distinct) and positive integers $\alpha_1,\alpha_2,\dots,\alpha_m$ such that
\[
    \mathcal{M} \cong R^r\oplus R/(p_1^{\alpha_1})\oplus R/(p_2^{\alpha_2})\oplus\dots\oplus R/(p_m^{\alpha_m}).
\] 
    Moreover, this representation is unique up to isomorphism.
\end{theorem}

Let $A \in M_n(\mathbb{F}_q)$. We define an $R = \mathbb{F}_q[x]$-module structure on $\mathcal{M}_A = \mathbb{F}_q^n$, where the action of $a \in \mathbb{F}_q \subset R$ and $x$ on $v \in \mathbb{F}_q^n$ is given by $a \cdot v = av$ and $x \cdot v = Av$, respectively. 

Recall that a \emph{free} $R$-module is an $R$-module that admits an $R$-basis (equivalently, it is isomorphic to a direct sum of copies of $R$), while a \emph{torsion} $R$-module is one in which every element is annihilated by some nonzero element of $R$. 
Any free $R$-module behaves like an infinite-dimensional vector space over $\mathbb{F}_q$. However, since $\mathcal{M}_A$ is finite-dimensional as a vector space over $\mathbb{F}_q$, its free rank must be zero. Therefore, $\mathcal{M}_A$ is a torsion $R$-module. This observation leads to the following corollary.

\begin{corollary}\label{cor:fun thm mod PID}
The module structure on $\mathcal{M}_A=\mathbb{F}_q^n$ defined above has a decomposition of the form 
\[
\mathcal{M}_A\cong \mathbb{F}_q[x]/(f_1(x)^{\alpha_1})\oplus \mathbb{F}_q[x]/(f_2(x)^{\alpha_2})\oplus\dots\oplus \mathbb{F}_q[x]/(f_m(x)^{\alpha_m})
\]
where each $f_i(x)$'s are irreducible polynomials (not necessarily distinct) and $\alpha_i$'s are positive integers.
\end{corollary}

\section{Positive definite matrices}\label{sec:posdef}

\subsection{Type 1}
\label{sec:posdef_1}

\begin{definition}
    A symmetric matrix $A\in \mathrm{GL}_n(\mathbb{F}_q)$ is said to be \textit{positive definite of type 1} (denoted by $\PD_1$) if each leading principal minor of $A$ is positive; that is, $\det(A_{[m]})$ is positive for all $m\in [n]$. 
    We will denote the set of $\PD_1$ matrices in $\mathrm{GL}_n(\mathbb{F}_q)$ by $\PD_1(n)$ and let $\pd_1(n)=|\PD_1(n)|$.
\end{definition} 

The following theorem gives the number of matrices in $\PD_1(n)$.

\begin{theorem}[{\cite[Corollary 1.5]{Vishwakarma2025}}]
\label{cor:num posdef_1}
The number of matrices in $\PD_1(n)$ is given by 
\[
\pd_1(n) = 
\begin{cases} 
\displaystyle
q^{\frac{n^2 - n}{2}}(q - 1)^n 
& \text{if }q\text{ is even}, \\[1em]

\displaystyle
q^{\frac{n^2 - n}{2}} \left( \dfrac{q - 1}{2} \right)^n
& \text{if }q\text{ is odd}.
\end{cases}
\]
\end{theorem}

\begin{remark}
Suppose we remove the requirement of symmetry from a $\PD_1$ matrix. 
Then we can inductively build such matrices of order $n$ 
by writing them as
\[
B=
\begin{pmatrix}
\begin{array}{c|c}
A & v_2\\
\hline
\\[-1em]
v_1^T &a_2
\end{array}
\end{pmatrix},
\]
where $A$ is a similar matrix of order $n-1$, 
$v_1, v_2\in \mathbb{F}_q^{n-1}$ and $a_2\in \mathbb{F}_q$. 

If $g(n)$ denotes the number of matrices in $M_n(\mathbb{F}_q)$ all of whose leading principal minors are positive, then $g(1) = |\mathbb{F}_q^{+}|$, and for $n \geq 2$, one can show that $g(n)$ satisfies the recurrence
\[
g(n)= q^{2n-2}\ |\mathbb{F}_q^{+}|\ g(n-1).
\]
This gives us the formula,
\[
g(n) = 
\begin{cases} 
\displaystyle
q^{n^2 - n}(q - 1)^n
& \text{if }q\text{ is even}, \\[1em]
\displaystyle
q^{n^2 - n} \left( \dfrac{q - 1}{2} \right)^n
& \text{if }q\text{ is odd}.
\end{cases}
\]
\end{remark} 

\subsection{Type 2}
\label{sec:posdef 2}

We extend the Cholesky decomposition to a definition here.

\begin{definition}\label{def:posdef 2}
    A symmetric matrix $A\in \mathrm{GL}_n(\mathbb{F}_q)$ is said to be \textit{positive definite of type 2} (denoted by $\PD_2$), if $A = LL^T$ for some lower triangular matrix $L\in \mathrm{GL}_n(\mathbb{F}_q)$ with positive diagonal entries. 
    We will denote the set of $\PD_2$ matrices in $\mathrm{GL}_n(\mathbb{F}_q)$ by $\PD_2(n)$ and let $\pd_2(n)=|\PD_2(n)|$.
\end{definition}

For definite fields, we have the following result.

\begin{theorem}[{\cite[Theorem 1.2]{CooperHannaWhitlatch2022}}]
\label{thm:posdef 1 and posdef 2 equiv def field}
If $\mathbb{F}_q$ is a definite field, a symmetric matrix in $\mathrm{GL}_n(\mathbb{F}_q)$ is $\PD_1$ if and only if it is $\PD_2$. Thus the number of matrices in $\PD_2(n)$ is given by
\[
\pd_2(n)=\begin{cases} 
\displaystyle
q^{\frac{n^2 - n}{2}}(q - 1)^n 
& \text{if }q\text{ is even}, \\

\displaystyle
q^{\frac{n^2 - n}{2}} \left( \dfrac{q - 1}{2} \right)^n
& \text{if }q\equiv 3\pmod{4}.
\end{cases}
\]
\end{theorem}

We now derive a formula for the number of matrices in $\PD_2(n)$ over an
indefinite finite field. For this, we first prove the following lemma.

\begin{lemma}
\label{lem:posdef_2 lem indef field}
    Let $\mathbb{F}_q$ be any finite field and $L_1,L_2\in \mathrm{GL}_n(\mathbb{F}_q)$ be two lower triangular matrices with nonzero diagonal entries. Then $L_1L_1^T=L_2L_2^T$, if and only if $L_2=L_1R_n$ where $R_n\in \mathrm{GL}_n(\mathbb{F}_q)$ is a diagonal matrix with diagonal entries 1 or -1.
\end{lemma}

\begin{proof}
    To prove the forward implication, we will use induction on $n$. For $n=1$, the statement is trivially true. Suppose the statement is true whenever the order of lower triangular matrices is $n-1$. 
    
Let $L_1=(\ell_{i,j}),\,L_2=(m_{i,j}) \in \mathrm{GL}_n(\mathbb{F}_q)$ and suppose that $L_1L_1^T=L_2L_2^T$.
By the induction hypothesis, there exists $
R_{n-1}=\operatorname{diag}(r_1,r_2,\dots,r_{n-1})\in \mathrm{GL}_{n-1}(\mathbb{F}_q),$ with diagonal entries $1$ or $-1$, such that
$L_{1\,[n-1]}=L_{2\,[n-1]}R_{n-1}$.
Consequently,
\begin{equation}
\label{eqn:posdef 2 low trian mat 1}
\ell_{i,j}=m_{i,j}r_j, \text{ for all } i,j\in [n-1].
\end{equation}
Comparing the $(n,j)$'th entries of $L_1L_1^T$ and $L_2L_2^T$, we obtain
\[
\sum_{i=1}^j\ell_{n,i}\ell_{j,i}
=
\sum_{i=1}^jm_{n,i}m_{j,i}.
\]
Hence,
\begin{equation}\label{eqn:posdef 2 low trian mat 2}
m_{n,j} =
(m_{j,j})^{-1}
\left(
\sum_{i=1}^j\ell_{n,i}\ell_{j,i}
-
\sum_{i=1}^{j-1}m_{n,i}m_{j,i}
\right).
\end{equation}
We claim that for every $i\in [n-1]$, we have
\[
m_{n,i}=\ell_{n,i}r_i.
\]

We prove the claim by strong induction on $j$. For $j=1$, \eqref{eqn:posdef 2 low trian mat 1} and \eqref{eqn:posdef 2 low trian mat 2} gives
\[
m_{n,1}=(m_{1,1})^{-1}\ell_{n,1}\ell_{1,1}=\ell_{n,1}r_1.
\]
Thus the base case holds.
Assume that the claim holds for all indices less than $j$. Then, by \eqref{eqn:posdef 2 low trian mat 2},
\[
\begin{aligned}
m_{n,j}&=m_{j,j}^{-1}
\left(\sum_{i=1}^j\ell_{n,i}\ell_{j,i}-\sum_{i=1}^{j-1}m_{n,i}m_{j,i}\right)\\
&=(r_j\ell_{j,j})^{-1}
\left(\sum_{i=1}^j\ell_{n,i}\ell_{j,i}-\sum_{i=1}^{j-1}r_i^2\ell_{n,i}\ell_{j,i}
\right)\\
&=(r_j\ell_{j,j})^{-1}\ell_{n,j}\ell_{j,j}=r_j\ell_{n,j},
\end{aligned}
\]
where the second equality follows from the induction hypothesis and \eqref{eqn:posdef 2 low trian mat 1}. Hence the claim follows by strong induction.
Taking $j=n$ in \eqref{eqn:posdef 2 low trian mat 2}, we obtain
\[
\begin{aligned}
m_{n,n}^2 &=\sum_{i=1}^n\ell_{n,i}^2-\sum_{i=1}^{n-1}m_{n,i}^2\\
&=\ell_{n,n}^2+\sum_{i=1}^{n-1}m_{n,i}^2r_i^2-\sum_{i=1}^{n-1}m_{n,i}^2=\ell_{n,n}^2,
\end{aligned}
\]
where the second equality follows from the claim. Therefore, $m_{n\times n}=\ell_{n\times n}$ or $m_{n\times n}=-\ell_{n\times n}$. In the former case, take $r_n=1$ and in the later case, take $r_n=-1$. Define $R_n=\operatorname{diag}(r_1,r_2,\dots,r_n)$, we will get $L_1=L_2R_n$. Hence by induction, the forward implication is proved.

The reverse implication is obvious, that is $L_1L_1^T=L_2L_2^T$ whenever $R_n\in GL_n(\mathbb{F}_q)$ is a diagonal matrix with diagonal entries 1 or -1.
\end{proof}

\begin{theorem}\label{thm:posdef 2 indef field}
For an indefinite field $\mathbb{F}_q$, the number of matrices in $PD_2(n)$ is given by
\[
\pd_2(n)=q^{\frac{n^2-n}{2}}\left(\frac{q-1}{4}\right)^n.
\]
\end{theorem}

\begin{proof}
Let $\mathcal{L}_n^+(\mathbb{F}_q)\subseteq \mathrm{GL}_n(\mathbb{F}_q)$ be set of all the lower triangular matrices in $\mathrm{GL}_n(\mathbb{F}_q)$ with positive diagonal entries. Note that
\[
|\mathcal{L}_n^+(\mathbb{F}_q)|=q^{\frac{n^2-n}{2}}\left(\frac{q-1}{2}\right)^n.
\]
Define a relation $\sim$ on $\mathcal{L}_n^+(\mathbb{F}_q)$ by setting $L_1\sim L_2$ if $L_1L_1^T=L_2L_2^T$. It is easy to see that $\sim$ is an equivalence relation. By \cref{lem:posdef_2 lem indef field}, each equivalence class contains exactly $2^n$ elements.
Therefore, the number of equivalence classes, which is also the number of matrices in $\PD_2(n)$ is 
\[
q^{\frac{n^2-n}{2}}\left(\frac{q-1}{4}\right)^n.
\]
This completes the proof.
\end{proof}

\begin{remark}
The proof of the \cref{thm:posdef 2 indef field} can be seen as an application of Lagrange's theorem. The set $\mathcal{L}_n^+(\mathbb{F}_q)$ is a subgroup of $\mathrm{GL}_n(\mathbb{F}_q)$, and the set of all such $R_n$'s defined in \cref{lem:posdef_2 lem indef field} is a subgroup of $\mathcal{L}_n^+(\mathbb{F}_q)$; call this subgroup $H$. We can rephrase \cref{lem:posdef_2 lem indef field} as follows: For two matrices 
$M,M'\in \mathcal{L}_n^+(\mathbb{F}_q)$, we have $MM^{T} = M'M'^{T}$ if and only if $M$ and $M'$ lie in the same left coset of the subgroup $H$. Therefore, the number of matrices in $\PD_2(n)$ equals the number of left cosets of $H$ in $\mathcal{L}_n^+(\mathbb{F}_q)$, which can be computed using Lagrange’s theorem.
\end{remark}

\subsection{Type 3}
\label{sec:posdef_3}

\begin{definition}
A symmetric matrix $A\in \mathrm{GL}_n(\mathbb{F}_q)$ is said to be \textit{positive definite matrix of type 3} (denoted by $\PD_3$), if there exists an invertible matrix $M\in \mathrm{GL}_n(\mathbb{F}_q)$ such that $A=MM^T$. We will denote the set of $\PD_3$ matrices in $\mathrm{GL}_n(\mathbb{F}_q)$ by $\PD_3(n)$ and let $\pd_3(n)=|\PD_3(n)|$.
\end{definition}

\begin{remark}\label{rem:posdef_2 and posdef_3}
Observe that, for every $n\geq 1$, we have $\PD_2(n)\subseteq \PD_3(n)$.

\end{remark}

The following lemma characterizes the matrices over field of characteristic two, that can be written in the form $MM^T$ where $M\in M_n(\mathbb{F}_q)$.

\begin{lemma}[{\cite[Theorem 7]{Albert1938}}]
\label{lem:posdef 3 q even} 
Over a field $\mathbb{F}_q$ of characteristic two, a nonzero symmetric matrix $A\in M_n(\mathbb{F}_q)$ can be written in the form $MM^T$ for some $M_n(\mathbb{F}_q)$ if and only if at least one of the diagonal elements of $A$ is nonzero.
\end{lemma}

The following lemma characterizes $\PD_3$ matrices over fields of odd characteristic.

\begin{lemma}[\cite{Dickson1901}]\label{lem:posdef 3 q odd} Over a field $\mathbb{F}_q$ of odd characteristic, an invertible symmetric matrix can be written in the form $MM^T$ if and only if the determinant of $A$ is a nonzero square.
\end{lemma}
 
Recall that $N(n,r)$ is the number of symmetric matrices of rank $r$ in $M_n(\mathbb{F}_q)$ given in \cref{thm:symm rank conunt}. Using \cref{lem:posdef 3 q even} and \cref{lem:posdef 3 q odd}, MacWilliams \cite{Macwilliams1969} calculated the number of matrices in $\PD_3(n)$. 

\begin{theorem}[\cite{Macwilliams1969}] \label{thm:posdef_3 count}
The number of matrices in $\PD_3(n)$ is given by
\[
\pd_3(n) = 
\begin{cases}
\displaystyle 
\prod_{i=0}^{(n-1)/2} \left(q^{n}-q^{2i}\right)
& \text{if $q$ is even and $n$ is odd}, \\[1.0em]

\displaystyle
(q^{n} - 1)\prod_{i=0}^{n/2-1} (q^{n-1} - q^{2i})
& \text{if $q$ is even and $n$ is even}, \\[1.0em]

\displaystyle
\frac{1}{2} N(n,n),
& \text{if $q$ is odd and $n$ is odd} \\[1.2em]

\displaystyle
\frac{1}{2}\,
\frac{q^{n/2} + 1}{q^{n/2}} N(n,n)
& \text{if $q\equiv 1 \pmod 4$ and $n$ is even}, \\[1.2em]

\displaystyle
\frac{1}{2}\,
\frac{q^{n/2} + (-1)^{n/2}}{q^{k}} N(n,n)
& \text{if $q\equiv 3 \pmod 4$ and $n$ is even}.
\end{cases}
\]
\end{theorem}

\subsection{Type 4}
\label{sec:posdef_4}

\begin{definition}
A symmetric matrix $A \in \mathrm{GL}_n(\mathbb{F}_q)$ is said to be \textit{positive definite of type 4} (denoted by $\PD_4$) if there exists an orthogonal matrix $U\in \mathrm{O}_n(\mathbb{F}_q)$ and a diagonal matrix $D\in \mathrm{GL}_n(\mathbb{F}_q)$ such that $A=U^{-1}DU$. 
We will denote the set of $\PD_4$ matrices in $\mathrm{GL}_n(\mathbb{F}_q)$ by $\PD_4(n)$ and let $\pd_4(n)=|\PD_4(n)|$.
\end{definition}

\begin{remark}\label{rem:posdef 4 and posdef 3 rel}
In the above definition, observe that $D$ can be written as $\operatorname{diag}(a_1^2,\dots,a_n^2)$, where $a_i\in \mathbb{F}_q^{\times}$ for all $i\in [n]$. Let $\sqrt{D}=\operatorname{diag}(a_1,\dots,a_n)$. Then
\[
    A=U^{-1}DU
    =U^{-1}\sqrt{D}\sqrt{D}^{T}{U^{-1}}^{T}
    =U^{-1}\sqrt{D}(U^{-1}\sqrt{D})^{T}.
\]
Hence, $A$ is a $PD_3$ matrix. Therefore, $\PD_4(n)\subseteq\PD_3(n)$ for all $n\geq 1$.
\end{remark}

We will say two matrices $A,B \in M_n(\mathbb{F})$ are \textit{orthogonally similar} if $B = U^{-1}AU$ for some orthogonal matrix $U \in \mathrm{O}_n(\mathbb{F})$. More generally, if $\mathcal{G}$ is a subgroup of $\mathrm{GL}_n(\mathbb{F})$, then two matrices $A,B \in M_n(\mathbb{F})$ are said to be \textit{$\mathcal{G}$-similar} if $B = X^{-1}AX$ for some $X \in \mathcal{G}$. Thus, a matrix is $\PD_4$ if and only if it is orthogonally similar to a diagonal matrix with positive diagonal entries.

\begin{remark}
Since $1$ is the only positive element in both $\mathbb{F}_2$ and $\mathbb{F}_3$, the identity matrix is the only diagonal matrix whose diagonal entries are all positive. Hence, for these fields, the identity matrix is the unique matrix in $\mathrm{PD}_4(2)$ and $\mathrm{PD}_4(3)$.
\end{remark}  

To determine the number of matrices in $\PD_4(n)$ over $\mathbb{F}_q$, we examine the conjugation action of $\mathrm{O}_n(\mathbb{F}_q)$ on $M_n(\mathbb{F}_q)$. Specifically, we study the equivalence classes of matrices under orthogonal similarity. Recall the $R=\mathbb{F}_q[x]$-module structure on $M_A=\mathbb{F}_q^n$ defined in \cref{sec:fg_mod_PID}, where the action of $a\in\mathbb{F}_q\subset R$ and $x$ on $v\in\mathbb{F}_q^n$ is given by $a\cdot v=av$ and $x\cdot v=Av$, respectively.

We first introduce some notation and definitions needed for the proof. Let $\mathcal{I}$ denote the set of irreducible polynomials in $\mathbb{F}_q[x]$, and let $\mathcal{P}$ denote the set of all partitions of positive integers.

Fix a matrix $A\in M_n(\mathbb{F}_q)$. Define a map $\varphi_A:\mathcal{I}\to\mathcal{P}$ by assigning to each irreducible polynomial $f(x)$ the partition $\lambda^A(f(x))=(\lambda_1^A(f(x)),\lambda_2^A(f(x)),\dots)$ such that the decomposition of $M_A$ in \cref{cor:fun thm mod PID} can be written as
\begin{equation}
M_A\cong\bigoplus_{f\in \mathcal{I}}\bigoplus_{i\geq 1}\mathbb{F}_q[x]/(f(x)^{\lambda_i^A(f(x))}).
\end{equation}
For an irreducible polynomial $f\in\mathcal{I}$, let $V_f \subseteq \mathbb{F}_q^n$ denote the set of vectors annihilated by some power of $f$, namely
\[
V_f=\{v\in\mathbb{F}_q^n\mid f(A)^r v=0 \text{ for some } r\geq 1\}.
\]
Since $v\in V_f$ implies $Av\in V_f$, each $V_f$ is an $A$-invariant subspace. If $f_1(x),\dots,f_m(x)$ are the irreducible polynomials appearing in the decomposition of $M_A$, then
\[
\mathbb{F}_q^n=\bigoplus_{i=1}^{m}V_{f_i}.
\]
For $f\in\mathcal{I}$ and $B\in M_n(\mathbb{F}_q)$, define the linear transformation $B_f$ by
\[
B_f(v)=
\begin{cases}
Bv & \text{if } v\in V_f,\\
0 & \text{if } v\notin V_f.
\end{cases}
\]

The following lemma characterizes the commutativity of two linear transformations in terms of their restrictions to the primary components.

\begin{lemma}\label{lem:mat comm A}
Let $f_1(x),\dots,f_m(x)$ be the irreducible polynomials appearing in the decomposition of $M_A$. Then $A$ commutes with $S$ if and only if $A_{f_i}$ commutes with $S_{f_i}$ for all $i\in[m]$.
\end{lemma}

\begin{proof}
Suppose that $A$ commutes with $S$. If $v\in V_{f_i}$, then
\[
f_i(A)^r(Sv)=Sf_i(A)^r(v)=0
\]
for every $r\geq 1$. Hence $Sv\in V_{f_i}$, and therefore $S(V_{f_i})\subseteq V_{f_i}$ for all $i$. It follows that $S$ preserves the same primary decomposition of $\mathbb{F}_q^n$ as $A$. Consequently,
$S_{f_i}(v_j)=0$ whenever $i\neq j$.

Now let $v=\sum_{i=1}^{m}v_i$, where $v_i\in V_{f_i}$ for each $i$. Since $S_{f_i}(v_j)=0$ for $j\neq i$, we have
\[
A_{f_i}S_{f_i}(v)=A_{f_i}S_{f_i}(v_i)=AS(v_i) =SA(v_i) =S_{f_i}A_{f_i}(v_i)=S_{f_i}A_{f_i}(v).
\]
Therefore, $A_{f_i}$ commutes with $S_{f_i}$ for every $i$.

Conversely, suppose that $A_{f_i}S_{f_i}=S_{f_i}A_{f_i}$ for all $i\in [m]$. Then
\[
AS(v)=\sum_{i=1}^{m}A_{f_i}S_{f_i}(v) =\sum_{i=1}^{m}S_{f_i}A_{f_i}(v) =SA(v).
\]
Since this holds for every $v\in\mathbb{F}_q^n$, we conclude that $AS=SA$.
\end{proof}

The following lemma was proved by Kung in the case $\mathcal{G} = \mathrm{GL}_n(\mathbb{F}_q)$. Using similar ideas, we generalize it to any subgroup $\mathcal{G}$ of $\mathrm{GL}_n(\mathbb{F}_q)$.

\begin{lemma}[{\cite{Kung1981}}]
\label{lem:kung thm}
    Let $\mathcal{G}$ be a subgroup of $\mathrm{GL}_n(\mathbb{F}_q)$ and let $A \in M_n(\mathbb{F}_q)$. If $f_1(x), \dots, f_m(x)$ are the polynomials appearing in \cref{cor:fun thm mod PID}, then the number of matrices in $M_n(\mathbb{F}_q)$ that are $\mathcal{G}$-similar to $A$ is given by
\[
    \frac{|\mathcal{G}|}{\prod_{i=1}^m c_{\mathcal{G}}(f_i(x), \varphi_A(f_i))}.
\]
    Here $c_{\mathcal{G}}(f(x), \lambda)$ denotes the number of matrices in $\mathcal{G}$ that commute with the matrix $M_f \in M_{n_f}(\mathbb{F}_q)$ (where $n_f=|\lambda^A(f(x))|\operatorname{deg}(f(x))$) satisfying:
\[
    \varphi_{M_f}(g) = 
    \begin{cases}
        \lambda^A(f(x)) & \text{if } g(x) = f(x), \\
        \varnothing & \text{if } g(x) \neq f(x).
    \end{cases}
\]
\end{lemma}

\begin{proof}
Recall that the orbit-stabilizer theorem states that if a group $\mathcal{G}$ acts on a set $X$, then for any $x \in X$, we have 
\[
|\mathcal{G}x| = \frac{|\mathcal{G}|}{|\mathcal{G}_x|},
\]
where $\mathcal{G}x = \{g \cdot x \mid g \in \mathcal{G}\}$ is the orbit of $x$, and $\mathcal{G}_x = \{g \in \mathcal{G} \mid g \cdot x = x\}$ is the stabilizer of $x$.

Consider the conjugation action of $\mathcal{G}$ on the set $X = M_n(\mathbb{F}_q)$, that is, for $X = M_n(\mathbb{F}_q)$ and $S\in \mathcal{G}$, $S \cdot A = SAS^{-1}$. Then $S \in \mathcal{G}_A$ if and only if $S$ and $A$ commute. Therefore, in this case, the number of matrices $\mathcal{G}$-similar to $A$, is given by
\begin{equation}\label{eq:posdef 4 orbit stab}
    |\mathcal{G}A| = \frac{|\mathcal{G}|}{|c(A)|},
\end{equation}
where $c(A)$ is the set of matrices in $\mathcal{G}$ that commute with $A$.

We now need to compute the the number of matrices in $\mathcal{G}$, which commute with $A$. By \cref{lem:mat comm A}, we have
\[
|c(A)| = \prod_{i=1}^m c_{\mathcal{G}}(f_i, \varphi_A(f_i)).
\]
Substituting this into \eqref{eq:posdef 4 orbit stab}, we obtain
\[
|\mathcal{G}A| = \frac{|\mathcal{G}|}{\prod_{i=1}^m c_{\mathcal{G}}(f_i(x), \varphi_A(f_i))},
\]
completing the proof.
\end{proof}

Consider the quotient group of $\mathrm{O}_n(\mathbb{F}_q)$ over $\mathrm{GL}_n(\mathbb{F}_q)$. Each matrix $M_1\in \mathrm{GL}_n(\mathbb{F}_q)$ can be associated with a symmetric matrix $MM^T$. This give a bijective correspondence between the quotient group and the set of symmetric matrices which can be written in the form $MM^T$ for some $M\in \mathrm{GL}_n(\mathbb{F}_1)$. Using this idea together with \cref{thm:posdef_3 count}, we can calculate the number of orthogonal matrices in $\mathrm{GL}_n(\mathbb{F}_q)$. 

\begin{theorem}(\cite{Macwilliams1969})\label{thm:orth matrices count}
The number of orthogonal matrices in $\mathrm{GL}_n(\mathbb{F}_q)$ is
\[
|\mathrm{O}_n(\mathbb{F}_q)|= 
\begin{cases}
\displaystyle
q^{k}\prod_{i=0}^{k-1} \left(q^{n-1}-q^{2i}\right)
& \text{if $q$ is even and $n=2k+1$}, \\[1em]

\displaystyle
q^{k}\prod_{i=1}^{k-1} \left(q^{n}-q^{2i}\right)
& \text{if $q$ \text{is even and} $n=2k$}, \\[1em]

\displaystyle
2q^k\prod_{i=0}^{k-1} (q^{n-1}-q^{2i})
& \text{if $q$ is odd and $n=2k+1$}, \\[1em]

\displaystyle
2(q^k-1)\prod_{i=1}^{k-1} (q^{n}-q^{2i})
& \text{if $q\equiv 1 \pmod 4$ and $n=2k$}, \\[1em]

\displaystyle
2(q^k+(-1)^{k+1})\prod_{i=1}^{k-1} (q^{n}-q^{2i})
& \text{if $q\equiv 3 \pmod 4$ and $n=2k$}.
\end{cases}
\]
\end{theorem}

Now, we are ready to give the formula for the number of matrices in $\PD_4(n)$. 

\begin{theorem}\label{thm:posdef 4 count}
The number of matrices in $\PD_4(n)$ is given by
\[
\pd_4(n)=\sum_{\substack{k_1 + k_2 + \dots + k_s = n \\ k_i \geq 0}} 
\frac{|\mathrm{O}_n(\mathbb{F}_q)|}{|\mathrm{O}_{k_1}(\mathbb{F}_q)||\mathrm{O}_{k_2}(\mathbb{F}_q)|\dots|\mathrm{O}_{k_s}(\mathbb{F}_q)|},
\]
where $s=|\mathbb{F}_q^{+}|$. 
\end{theorem}

\begin{proof}
We want to find the matrices which are orthogonally similar to a diagonal matrix with positive diagonal entries. For a fixed diagonal matrix 
\[D=\operatorname{diag}(\underbrace{a_1, \dots, a_1}_{k_1 \text{ times}},\underbrace{a_2, \dots, a_2}_{k_2 \text{ times}},\dots,\underbrace{a_r, \dots, a_r}_{k_r \text{ times}})\in \mathrm{GL}_n(\mathbb{F}_q),
\] 
where $a_i$'s are distinct positive elements
\[
\varphi_D(f) = 
    \begin{cases}
        (1^{k_i}) & \text{if } f = x - a_i, \\
        \varnothing & \text{otherwise}.
    \end{cases}
\]
Note that every matrix in $\mathrm{O}_{k_i}(\mathbb{F}_q)$ commutes with $\operatorname{diag}(\underbrace{a_1, \dots, a_1}_{k_i \text{ times}})$. Therefore, in this case, $c_{\mathcal{G}}(f_i, \varphi_D(f_i)) = |\mathrm{O}_{k_i}(\mathbb{F}_q)|$, where $f_i(x)=x-a_i$. Applying \cref{lem:kung thm} with $\mathcal{G}= \mathrm{O}_n(\mathbb{F}_q)$, the number of matrices orthogonal similar to $D$ is
\[
\frac{|\mathrm{O}_n(\mathbb{F}_q)|}{|\mathrm{O}_{k_1}(\mathbb{F}_q)||\mathrm{O}_{k_2}(\mathbb{F}_q)|\cdots|\mathrm{O}_{k_r}(\mathbb{F}_q)|}
\]
Taking summation over all the diagonal matrix will give the desire result.
\end{proof}

\begin{corollary}\label{cor:posdef 4 gen func}
The sequence $(\pd_4(n))_{n\geq 0}$ satisfies the generating function-type equation,
\[
\sum_{n\geq0}\pd_4(n)\frac{x^n}{|\mathrm{O}_n(\mathbb{F}_q)|}=\left(\sum_{k\geq 0}\frac{x^k}{|\mathrm{O}_k(\mathbb{F}_q)|}\right)^s,
\]
where $s=|\mathbb{F}_q^{+}|$.
\end{corollary}

\subsection{Relating these definitions}
\label{sec:rel posdef}

It is evident that the equivalence criteria for positive definite matrices over the real and complex fields given in \cref{thm:posdef RC} are in general not equivalent over finite fields.
In this section, we examine how two classically equivalent criteria of positive definiteness behave when considered over a finite field. We will do that by investigating if there exists any relationship between the corresponding sets of such matrices.
The Venn diagrams in \cref{fig:pd_relations} summarize the relations between the sets of positive definite matrices of various types\footnote{It should follow from \cite[Theorem 23]{CooperHannaWhitlatch2022} that $\PD_1, \PD_2$ and $\PD_3$ matrices are the same for definite fields. But there is a mistake in their argument because the Gram matrix formed by linearly independent vectors, which is a $\PD_3$ matrix, need not be a $\PD_1$ matrix.}. 
The justification is below.

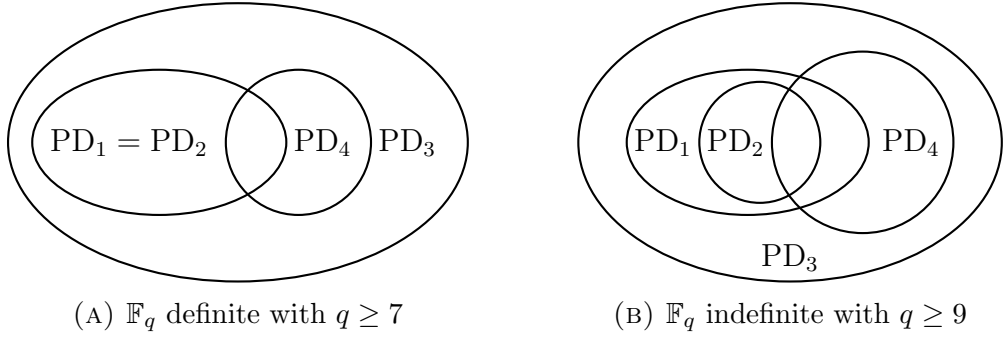
\begin{figure}[h]
\centering

\begin{subfigure}{0.45\textwidth}
\centering

\begin{tikzpicture}[scale=0.8]


\draw[thick] (-4.5,0) ellipse (3.8 and 2.3);
\node at (-1.7,0) {$\PD_{3}$};

\draw[thick] (-5.8,0) ellipse (2.1 and 1.2);
\node at (-6.3,0) {$\PD_{1}=\PD_{2}$};

\draw[thick] (-3.5,0) circle (1.2);
\node at (-3.1,0) {$\PD_{4}$};

\end{tikzpicture}
\caption{$\mathbb{F}_q$ definite with $q\geq 7$}
\end{subfigure}
\begin{subfigure}{0.45\textwidth}
\centering
\begin{tikzpicture}[scale=0.8]

\draw[thick] (5,0) ellipse (3.5 and 2.3);
\node at (5,-1.9) {$\PD_{3}$};

\draw[thick] (4.3,0) ellipse (2.0 and 1.2);
\node at (2.9,0) {$\PD_{1}$};

\draw[thick] (4.5,0) circle (1);
\node at (4.1,0) {$\PD_{2}$};

\draw[thick] (6.2,0) circle (1.5);
\node at (7,0) {$\PD_{4}$};
\end{tikzpicture}
\caption{$\mathbb{F}_q$ indefinite with $q\geq 9$}
\end{subfigure}

\caption{Relations between the sets of positive definite matrices of various types over $\mathbb{F}_q$ for $n\geq 2$.} 

\label{fig:pd_relations}

\end{figure}

\subsubsection{$\PD_{1}(n)$ and $\PD_{2}(n)$:}
\label{sec:PD_1_PD_2_rel} 
From \cref{sec:posdef 2}, we know that $\PD_{1}(n)=\PD_{2}(n)$ when the field is definite, and that $\PD_{2}(n)\subseteq \PD_{1}(n)$ when the field is indefinite. 
We will now construct matrices in $\PD_1(n)\setminus\PD_2(n)$ when $\mathbb{F}_q$ is indefinite. By \cref{thm:x^2+y^2=1}, for $q\geq 7$, the equation $x^2+y^2=1$ has more than $4$ solutions. Therefore, there exists $(a,b)\in\mathbb{F}_q^2$ such that $a^2+b^2=1$ and $(a,b)\notin\{(\pm1,0),(0,\pm1)\}$. In particular, $a,b\neq0$. Let $c\in\mathbb{F}_q^{+}$ be such that $c$ does not have a positive square root in $\mathbb{F}_q$. By \cref{prop:php fin fld}(2), we can choose $d_1,d_2\in\mathbb{F}_q^-$ such that $c=d_1+d_2$. Consider the matrix
\[
\begin{aligned}
A_1 
={}&
\left(\begin{array}{c|c}
\begin{matrix}
a & b\\
-b & a
\end{matrix}
&
\mathbf{0}_{2\times(n-2)}
\\
\hline
\mathbf{0}_{(n-2)\times2}
&
\iden_{n-2}
\end{array}\right)
\left(\begin{array}{c|c}
\begin{matrix}
d_1 a^{-2} & 0\\
0 & d_2 b^{-2}
\end{matrix}
&
\mathbf{0}_{2\times(n-2)}
\\
\hline
\mathbf{0}_{(n-2)\times2}
&
\iden_{n-2}
\end{array}\right)
\left(\begin{array}{c|c}
\begin{matrix}
a & -b\\
b & a
\end{matrix}
&
\mathbf{0}_{2\times(n-2)}
\\
\hline
\mathbf{0}_{(n-2)\times2}
&
\iden_{n-2}
\end{array}\right)
\\
={}&
\left(\begin{array}{c|c}
\begin{matrix}
c & d_2 ab^{-1}-d_1 a^{-1}b\\
d_2 ab^{-1}-d_1 a^{-1}b &
d_1 a^{-2} b^2 + d_2 a^2 b^{-2}
\end{matrix}
&
\mathbf{0}_{2\times(n-2)}
\\
\hline
\mathbf{0}_{(n-2)\times2}
&
\iden_{n-2}
\end{array}\right).
\end{aligned}
\]
Then $A_1 \in\PD_1(n)$, since its first entry is positive and all its remaining leading principal minors are equal to its determinant, which is $d_1 d_2/(a^2 b^2)$ and is also positive. 
On the other hand, $A_1 \notin\PD_2(n)$, since $(A_1)_{\{1\}}=c$, which does not have a positive square root in $\mathbb{F}_q$. Therefore, $A_1 \in\PD_1(n)\setminus\PD_2(n)$.
    
\subsubsection{$\PD_{4}(n)$ with $\PD_{1}(n)$ and $\PD_{2}(n)$:}
\label{sec:PD_4_PD_1 2_rel} 
We will now show that there is no inclusion relation between $\PD_{1}(n)$ and $\PD_{4}(n)$, and consequently, no relation between $\PD_{2}(n)$ and $\PD_{4}(n)$ in $M_n(\mathbb{F}_q)$ when $q\geq 7$. 
    
First, let us construct examples of matrices in $\PD_{4}(n)\setminus\PD_{1}(n)$. Let $q$ be odd.  By \cref{thm:x^2+y^2=1}, for $q\geq 7$, we can choose $a_1,b_1\in\mathbb{F}_q$ such that $a_1^2+b_1^2=1$ and $(a_1,b_1)\notin\{(\pm1,0),(0,\pm1)\}$. Let $c_1\in \mathbb{F}_q^{-}$. By \cref{prop:php fin fld}(1), we can choose $d_3,d_4\in \mathbb{F}_q^\times$ such that $c_1=d_3^2+d_4^2$. 
Consider the matrix
\[
\begin{aligned}
A_2 =&\left(\begin{array}{c|c}
\begin{matrix}
a_1 & b_1  \\
-b_1 & a_1  
\end{matrix}
&
\textbf{0}_{2\times(n-2)} \\
\hline
\textbf{0}_{(n-2)\times 2} & \iden_{n-2}
\end{array}\right)
\left(\begin{array}{c|c}
\begin{matrix}
d_3^2a_1^{-2} & 0  \\
0 & d_4^2b_1^{-2}  
\end{matrix}
&
\textbf{0}_{2\times(n-2)} \\
\hline
\textbf{0}_{(n-2)\times 2} & \iden_{n-2}
\end{array}\right)
\left(\begin{array}{c|c}
\begin{matrix}
a_1 & -b_1  \\
b_1 & a_1  
\end{matrix}
&
\textbf{0}_{2\times(n-2)} \\
\hline
\textbf{0}_{(n-2)\times 2} & \iden_{n-2}
\end{array}\right)\\
=& \left(\begin{array}{c|c}
\begin{matrix}
c_1 & d_4^2a_1b_1^{-1}-d_3^{2}a_1^{-1}b_1  \\
d_4^2a_1b_1^{-1}-d_3^{2}a_1^{-1}b_1 & d_3^2a_1^{-2}b_1^{2}+d_4^{2}a_1^2b_1^{-2}  
\end{matrix}
&
\textbf{0}_{2\times(n-2)} \\
\hline
\textbf{0}_{(n-2)\times 2} & \iden_{n-2}
\end{array}\right)
\end{aligned}
\]
Since $\det ((A_2)_{\{1\}})=c_1$ is negative, $A_2 \in \PD_{4}(n)\setminus \PD_{1}(n)$.

Now consider the case when $q$ is even. Again, by \cref{thm:x^2+y^2=1}, for $q\geq8$, we can choose $(a_2,b_2)\in\mathbb{F}_q^2$ such that $a_2^2+b_2^2=1$ and $(a_2,b_2)\notin\{(\pm1,0),(0,\pm1)\}$. 
Consider the matrix
\[
\begin{aligned}
A_3={}&\left(\begin{array}{c|c}
\begin{matrix}
a_2 & b_2\\
-b_2 & a_2
\end{matrix}
&
\mathbf{0}_{2\times(n-2)}
\\
\hline
\mathbf{0}_{(n-2)\times2} & \iden_{n-2}
\end{array}\right)
\left(\begin{array}{c|c}
\begin{matrix}
a_2^{-2} & 0\\
0 & b_2^{-2}
\end{matrix}
&
\mathbf{0}_{2\times(n-2)}
\\
\hline
\mathbf{0}_{(n-2)\times2} & \iden_{n-2}
\end{array}\right)
\left(\begin{array}{c|c}
\begin{matrix}
a_2 & -b_2\\
b_2 & a_2
\end{matrix}
&
\mathbf{0}_{2\times(n-2)}
\\
\hline
\mathbf{0}_{(n-2)\times2} & \iden_{n-2}
\end{array}\right)
\\
={}&\left(\begin{array}{c|c}
\begin{matrix}
0 & a_2b_2^{-1}-a_2^{-1}b_2\\
a_2b_2^{-1}-a_2^{-1}b_2 & a_2^{-2}b_2^2+a_2^2b_2^{-2}
\end{matrix}
&
\mathbf{0}_{2\times(n-2)}
\\
\hline
\mathbf{0}_{(n-2)\times2} & \iden_{n-2}
\end{array}\right).
\end{aligned}
\]
Since $\det((A_3)_{\{1\}})=0$, we have $B\notin\PD_1(n)$. 
On the other hand, $A_3 \in\PD_4(n)$ by construction. 
Therefore, $A_3\in\PD_4(n)\setminus\PD_1(n)$.
    
Now, we will see construction of matrices in $\PD_{2}(n)\setminus \PD_{4}(n)$. First, consider the case when $q$ is odd. Choose $c_2 \in \mathbb{F}_q^{\times}$ such that $c_2^2+1\in \mathbb{F}_q^-$. Consider the matrix
\[
A_4=\left(\begin{array}{c|c}
\begin{matrix}
1 & 0  \\
2c_2 & 1  
\end{matrix}
&
\textbf{0}_{2\times(n-2)} \\
\hline
\textbf{0}_{(n-2)\times 2} & \iden_{n-2}
\end{array}\right)
\left(\begin{array}{c|c}
\begin{matrix}
1 & 2c_2  \\
0 & 1  
\end{matrix}
&
\textbf{0}_{2\times(n-2)} \\
\hline
\textbf{0}_{(n-2)\times 2} & \iden_{n-2}
\end{array}\right)=
\left(\begin{array}{c|c}
\begin{matrix}
1 & 2c_2  \\
2c_2 & 4c_2^2+1  
\end{matrix}
&
\textbf{0}_{2\times(n-2)} \\
\hline
\textbf{0}_{(n-2)\times 2} & \iden_{n-2}
\end{array}\right).
\]
The characteristic polynomial of $A_4$ is $(x-1)^{n-2}(x^2-(4c_2^2+2)x+1)$. The discriminant of the quadratic factor is $16c_2^2(c_2^2+1)$. Since this discriminant is not a square in $\mathbb{F}_q$, the quadratic factor does not split over $\mathbb{F}_q$. Consequently, $A_4$ is not diagonalizable over $\mathbb{F}_q$, and hence $A_4\in \PD_2(n)\setminus \PD_4(n)$.

Now consider the case when $q$ is even. To construct a matrix in $\PD_{2}(n)\setminus \PD_{4}(n)$, choose an irreducible quadratic monic polynomial $f(x)=x^2+a_3 x+b_3 \in \mathbb{F}_q[x]$. Observe that $a_3, b_3 \neq 0$. Let $c_3^2=b_3$. Consider the matrix
\[
A_5=\left(\begin{array}{c|c}
\begin{matrix}
a_3 & c_3 \\
c_3 & 0  
\end{matrix}
&
\textbf{0}_{2\times(n-2)} \\
\hline
\textbf{0}_{(n-2)\times 2} & \iden_{n-2}
\end{array}
\right).
\]
Note that the characteristic polynomial of $A_5$ is $(x-1)^{n-2}f(x)$, and hence it has $n-2$ eigenvalues in $\mathbb{F}_q$. Therefore, it is not diagonalizable and hence $A_5 \in \PD_{2}(n)\setminus \PD_{4}(n)$.

We now construct matrices in $\bigl(\PD_1(n)\cap\PD_4(n)\bigr)\setminus\PD_2(n)$ when $\mathbb{F}_q$ is indefinite. Choose positive elements $a_1,a_2,\dots,a_n\in\mathbb{F}_q^{+}$ such that $a_1$ does not have a positive square root in $\mathbb{F}_q$. Then the matrix $\operatorname{diag}(a_1,a_2,\dots,a_n)$ belongs to $\PD_1(n)\setminus\PD_2(n)$. Moreover, since it is a diagonal matrix, it also belongs to $\PD_4(n)$. Thus, $\operatorname{diag}(a_1,a_2,\dots,a_n)\in\bigl(\PD_1(n)\cap\PD_4(n)\bigr)\setminus\PD_2(n)$.

Lastly, the matrix $A_1$ constructed in \cref{sec:PD_1_PD_2_rel} provides an example of a matrix in $\PD_1(n)\setminus\bigl(\PD_2(n)\cup\PD_4(n)\bigr)$ when the field is indefinite.
    
\subsubsection{$\PD_{3}(n)$ with $\PD_{1}(n),\PD_{2}(n)$ and $\PD_{4}(n)$:}
\label{sec:PD_3_PD_1 2 4_rel} 
By \cref{rem:posdef 4 and posdef 3 rel} and \cref{rem:posdef_2 and posdef_3}, we have $\PD_{2}(n),\PD_{4}(n)\subseteq \PD_{3}(n)$. \cref{thm:posdef 1 and posdef 2 equiv def field} implies that $\PD_{1}(n) \subseteq \PD_{3}(n)$ for a definite field. Moreover, by \cref{lem:posdef 3 q odd}, $\PD_{1}(n) \subseteq \PD_{3}(n)$ when the field is indefinite. 

We now construct matrices in $\PD_3(n)\setminus\bigl(\PD_1(n)\cup\PD_2(n)\cup\PD_4(n)\bigr)$. First consider the case of $q$ odd. 
Let $a_4,a_5 \in\mathbb{F}_q^{-}$. 
The matrix $A_6=\operatorname{diag}(a_4,a_5,1,\dots,1)\in \PD_3(n)$ by \cref{lem:posdef 3 q odd}. 
However, since its $(1,1)$ entry and two of the eigenvalues are negative, 
$A_6 \notin \PD_1(n)\cup \PD_2(n)\cup \PD_4(n)$. 

Now consider the case when $q$ is even. 
Recall the irreducible quadratic monic polynomial $f$ from \cref{sec:PD_4_PD_1 2_rel} and let $c_3^2=b_3$ as before. Consider the matrix:
\[
A_7=\left(\begin{array}{c|c}
\begin{matrix}
0 & c_3\\
c_3 & a_3
\end{matrix}
&
\mathbf{0}_{2\times(n-2)}
\\
\hline
\mathbf{0}_{(n-2)\times 2}
&
\iden_{n-2}
\end{array}\right).
\]
Note that the characteristic polynomial of $A_7$ is $(x-1)^{n-2}f(x)$, and hence it has $n-2$ eigenvalues in $\mathbb{F}_q$. Therefore, $A_7$ is not diagonalizable and $A_7\notin\PD_4(n)$. Moreover, $\det((A_7)_{\{1\}})=0$, and hence $A_7 \notin\PD_1(n)=\PD_2(n)$. By \cref{lem:posdef 3 q even}, we have $A_7 \in\PD_3(n)$.  
Therefore, $A\in\PD_3(n)\setminus\bigl(\PD_1(n)\cup\PD_2(n)\cup\PD_4(n)\bigr)$.

\section{Positive semidefinite matrices}
\label{sec:psd}

In this section, we compute the number of matrices in $M_n(\mathbb{F}_q)$ satisfying each of the conditions in \cref{thm:psd RC}.

\subsection{Type 1} 
\label{sec:psd 1}

Motivated by \cref{cor:posdef type 1 defn real case 1}, we give our first definition of a positive semidefinite matrix.

\begin{definition}\label{def:posdef 1}

A symmetric matrix $A\in M_n(\mathbb{F}_q)$ is said to be \textit{positive semidefinite of type 1} (denoted $PSD_1$), if for each $m \in [n]$, there exists a subset $K_m \subseteq [m]$ such that the principal submatrix $A_{K_m}$ is $PD_1$, and whenever $M \subseteq [m]$ satisfies $|M|>|K_m|$, the principal submatrix $A_M$ is singular. 
We denote the set of $\PSD_1$ matrices in $M_n(\mathbb{F}_q)$ by $\PSD_1(n)$, and let $\psd_1(n,r)$ denote the number of rank $r$ matrices in $\PSD_1(n)$.

\end{definition}

It is evident from the definition that the rank of $A$ is at least $|K_n|$. Suppose the rank of $A$ is $r>|K_n|$. Then there exists a set $M\subseteq [n]$ of size $r$ such that $A_{M,[n]}$ and $A_{[n],M}$ are full rank matrices. But by \cref{lem:bocher thm} this means $A_M$ is nonsingular which contradicts the definition. Therefore, the rank of $A$ is $|K_n|$.

The following theorem reduces the number of principal submatrices that need to be checked in the definition of a $\PSD_1$ matrix.
The proof is identical to that of \cref{thm:psd princ submatrix} and is omitted.

\begin{theorem}
For a symmetric matrix $A\in M_n(\mathbb{F}_q)$, we have $A\in\PSD_1(n)$ if and only if, for each $m\in[n]$, there exists a subset
$K_m\subseteq[m]$ such that $A_{K_m}$ is $\PD_1$, and every principal
submatrix $A_M$ is singular whenever $K_m\subset M\subseteq[m]$ and
$|M|\in\{|K_m|+1,|K_m|+2\}$.
\end{theorem}

\begin{remark}
The equivalence of \cref{thm:psd RC}(f) and \cref{cor:posdef type 1 defn real case 1} over the complex field doesn't extend to finite fields. Consider the following matrix in $M_3(\mathbb{F}_3)$:
\[
A=\begin{pmatrix}
    2&2&0\\
    2&1&2\\
    0&2&2
\end{pmatrix}
\]
This matrix satisfies the condition in \cref{thm:posdef RC}(b), because $A_{\{2,3\}}$ is a $\PD_1$ matrix and $\det(A) = 0$. However, it does not satisfy the condition in \cref{cor:posdef type 1 defn real case 1}, as $A_{\{1\}}$ and $A_{\{1,2\}}$ do not satisfy the condition in \cref{thm:posdef RC}(b). 
\end{remark}

From \cref{def:posdef 1}, it follows that every leading principal submatrix of a $\PSD_1$ matrix is also $\PSD_1$.

We will find a recurrence satisfied by $\psd_1(n,r)$, solving which gives an explicit formula for the number of matrices in $\PSD_1(n)$ of rank $r$.

\begin{lemma}
\label{lem:psd 1 extn}
Let $A \in \PSD_1(n)$ be of rank $r$, and let $K_1,K_2 \subseteq [n-1]$ be such that $|K_1| = |K_2| = r$ and $A_{K_1}$ and $A_{K_2}$ are $\PD_1$ matrices. Let
\[
B=
\left(
\begin{array}{c|c}
A & v\\
\hline
\\[-1em]
v^T &a_2
\end{array}
\right),
\]
where $v^T$ linearly depends on the rows of $A$, and $a_2 \in \mathbb{F}_q$. Then $B_{K_2\cup\{n\}}$ is $\PD_1$ if and only if $B_{K_1\cup\{n\}}$ is $\PD_1$. Moreover, there are $\mathbb{F}_q^{+}$ choices of $a_2$ for which $B_{K_1\cup\{n\}}$ is $\PD_1$.

\end{lemma}

\begin{proof}
Observe that we can uniquely write $a_2 = a_2' + p$
such that the last row of
\[
\left(
\begin{array}{c|c}
A & v\\
\hline
\\[-1em]
v^T &a_2'
\end{array}
\right)
\]
linearly depends on the first $n-1$ rows. By multilinearity of the determinant,
\[
\det(B_{K_i\cup\{n\}})=\det\left(
\begin{array}{c|c}
A_{K_i} & v_{K_i} \\ \hline
\\[-1em]
v_{K_i}^T & a_2'
\end{array}
\right)+
\det
\left(
\begin{array}{c|c}
A_{K_i} & v_{K_i} \\ \hline
\mathbf{0}_{1\times |K_1|} & p
\end{array}
\right)
=
p\,\det(A_{K_i}),
\]
for $i\in\{1,2\}$. Therefore, for each $i \in \{1,2\}$, the matrix $B_{K_i\cup\{n\}}$ is $\PD_1$ if and only if $p \in \mathbb{F}_q^{+}$. Consequently, $B_{K_1\cup\{n\}}$ is $\PD_1$ if and only if $B_{K_2\cup\{n\}}$ is $\PD_1$. Moreover, since $p$ ranges over $\mathbb{F}_q^{+}$, there are $|\mathbb{F}_q^{+}|$ choices of $a_2$ for which $B_{K_1\cup\{n\}}$ is $\PD_1$. This completes the proof.
\end{proof}

\begin{theorem}\label{thm:psd 1 recur}
The number $\psd_1(n,r)$ of rank $r$ matrices in $\PSD_1(n)$ satisfies the equations $\psd_1(n,0)=1$, $\psd_1(n,r)=0$ whenever $r>n$ or $r<0$, and the following recurrence relation:
\[
    \psd_1(n,r)=\begin{cases}
\displaystyle
q^r\ \psd_1(n-1,r)+q^{r-1}\ \left(\frac{q-1}{2}\right)\ \psd_1(n-1,r-1)
& \text{if $q$ is odd}, \\[1em]

\displaystyle
q^r\ \psd_1(n-1,r)+q^{r-1}\ ({q-1})\ \psd_1(n-1,r-1)
& \text{if $q$ is even}.
\end{cases}
\]
\end{theorem}

\begin{proof}
Initial conditions for $\psd_1(n,r)$ follows trivially. Since all leading principal minors of a $\PSD_1$ matrix are again $\PSD_1$, we analyze in how many ways a matrix $A \in \PSD_1(n-1)$ can be extended to a rank $r$ matrix in $\PSD_1(n)$. For each $i\in[n-1]$, let $K_i\subseteq[i]$ be the subset associated with $A$ as in \cref{def:posdef 1}.

Let
\[
B=
\left(
\begin{array}{c|c}
A&v\\
\hline
\\[-1em]
v^T&a_2
\end{array}
\right),
\]
where $v\in \mathbb{F}_q^{n-1}$ and $a_2\in\mathbb{F}_q$.
We have three cases depending on the rank of $A$.

\medskip

\noindent\textbf{Case 1:} $\operatorname{rank}(A)=r$. By
\cref{lem:mac symm class}(1), there are $q^r$ symmetric matrices of rank $r$ in $M_n(\mathbb{F}_q)$ having $A$ as their leading principal submatrix.
For each such matrix $B$, define
$K'_i=K_i$ for every $i\in[n-1]$, and set
$K'_n=K_{n-1}$. Then the subsets
$K'_1,K'_2,\dots,K'_n$ satisfy
\cref{def:posdef 1} for $B$. Hence $B$ is a
$\PSD_1$ matrix. Therefore, the number of matrices in $\PSD_1(n)$ of rank $r$ whose leading principal submatrix of order $n-1$ also has rank $r$ is $q^r\psd_1(n-1,r)$.

\medskip

\noindent\textbf{Case 2:} $\operatorname{rank}(A)=r-1$. By \cref{lem:mac symm class}(2), the row vector $v^T$ must lie in the row space of $A$. Since $\operatorname{rank}(A)=r-1$, there are $q^{r-1}$ possible choices for $v$. Once $v$ is fixed, it remains to choose $a_2$ so that there exists a subset $K'_n\subseteq[n]$ of size $r$ for which the principal submatrix $B_{K'_n}$ is $PD_1$. By \cref{lem:psd 1 extn}, it suffices to count the choices of $a_2$ for which $B_{K_{n-1}\cup\{n\}}$ is $PD_1$. For each such choice, define $K'_i=K_i$ for $i\in[n-1]$ and $K'_n=K_{n-1}\cup\{n\}$. Then the subsets $K'_1,K'_2,\dots,K'_n$ satisfy the conditions in \cref{def:posdef 1}. Hence, $B\in\PSD_1(n)$, and $\operatorname{rank}(B)=r$.

By \cref{lem:psd 1 extn}, there are $|\mathbb{F}_q^{+}|$ choices for $a_2$. 
Consequently, there are $q^{r-1}|\mathbb{F}_q^{+}|$ matrices in $\PSD_1(n)$ of rank $r$ whose leading principal submatrix is $A$.
 Therefore, the number of $\PSD_1$ matrices of rank $r$ whose leading principal submatrix of order $n-1$ has rank $r-1$ is $q^{r-1}|\mathbb{F}_q^{+}|\,\psd_1(n-1,r-1)$.

\medskip

\noindent\textbf{Case 3:} $\operatorname{rank}(A)=r-2$. We claim that $A$ cannot be extended to a matrix in $\PSD_1(n)$ of rank $r$.
Suppose, for the sake of contradiction, that there exists $B\in \PSD_1(n)$ of rank $r$ whose leading principal submatrix of order $n-1$ is $A$. Since $B$ has rank $r$, there exists a subset $K \subseteq [n]$ of size $r$ such that the submatrix $B_K$ is $\PD_1$. However, the submatrix $B_{K \setminus \{n\}}$, which has size $r-1$, must then be nonsingular, contradicting the assumption that the leading principal submatrix of order $n-1$ has rank $r-2$. Therefore, this case contributes nothing.
\end{proof}

Solving the recurrence in \cref{thm:psd 1 recur}, we obtain the following formula.

\begin{corollary}
The number of matrices in $\PSD_1(n)$ of rank $r$ is given by 
\[
\psd_1(n,r)=\begin{cases}
\displaystyle
\qbinom{n}{r} q^{\frac{r(r-1)}{2}}\left(\frac{q-1}{2}\right)^r
& \text{if $q$ is odd}, \\[1em]

\displaystyle
\qbinom{n}{r}q^{\frac{r(r-1)}{2}}({q-1})^r
& \text{if $q$ is even}.
\end{cases}
\]
\end{corollary}
    
\subsection{Type 2}
\label{sec:psd 2}

\begin{definition}
A symmetric matrix $A\in M_n(\mathbb{F})$ is said to be a \textit{positive semidefinite matrix of type 2} (denoted by $\PSD_2$) if $A=LL^T$ for some lower triangular matrix $L$ whose diagonal entries are nonnegative. We denote the set of $\PSD_2$ matrices in $M_n(\mathbb{F}_q)$ by $\PSD_2(n)$, and let $\psd_2(n)=|\PSD_2(n)|$.
\end{definition} 

At present, we do not have a simple closed-form expression for the number of matrices in $\PSD_2(n)$. Instead, we derive an upper bound for this quantity. We will need the following classical identity. 

\begin{proposition}[$q$-binomial theorem {\cite[Section 1.8]{Stanley2011}}]
\label{prop:q bin thm}
We have
\[
\sum_{k=0}^{n+1} x^k q^{\binom{k}{2}} \qbinom{n+1}{k} = 
\prod_{i=0}^n(1+x q^i).
\]
\end{proposition}

\begin{theorem}
    The number of matrices in $\PSD_2(n)$ satisfies the following bound
\[
    \psd_2(n)\leq \begin{cases}
    \displaystyle
    2\prod_{i=1}^n\left(1+\frac{q^i}{2}\right)
    & \text{if $q\equiv 3\pmod{4}$}, \\[1.5em]

    \displaystyle
    \prod_{i=1}^n\left(1+\frac{q^i}{4}\right)
    & \text{if $q\equiv 1\pmod{4}$}, \\[1.5em]

    \displaystyle
    q^{\frac{n(n-1)}{2}}(q+1)^{n}
    & \text{if $q$ is even}. \\
    \end{cases}
\]
\end{theorem}

\begin{proof}
First, consider the case when $q$ is odd. 
Replacing $x$ by $xq$ and $n$ by $n-1$ in \cref{prop:q bin thm}, we obtain
\begin{equation}\label{eq:transf binom thm}
\sum_{k=0}^{n} x^k q^{\binom{k+1}{2}} \qbinom{n}{k}
= \prod_{i=1}^n(1+xq^i).
\end{equation}
Comparing the coefficient of $x^k$, we obtain
\begin{equation}\label{eq:q bin coeff}
    \sum_{1\leq i_1<\cdots<i_{k}\leq n}
    q^{i_1+\cdots+i_{k}}=q^{\binom{k+1}{2}} \qbinom{n}{k}.
\end{equation}
Given a composition $k_1+k_2+k_3=n$ of $n$, let $\mathcal{L}_{k_1,k_2,k_3}$ denote the set of lower triangular matrices in $M_n(\mathbb{F}_q)$ such that $k_1$ rows are nonzero rows with zero diagonal entries, $k_2$ rows have positive diagonal entries, and $k_3$ rows are zero rows. Then
\[
    |\mathcal{L}_{k_1,k_2,k_3}|=\sum_{\substack{
    I,J\subseteq [n]\\|I|=k_1,|J|=k_2,I\cap J=\varnothing
    }}
    \left(\prod_{m\in I}(q^{m-1}-1)\right)\left(\prod_{m\in J}q^{m-1}|\mathbb{F}_q^{+}|\right).
\]
    where $I$ denote the set of indices of the nonzero rows with zero diagonal entries, and $J$ denote the set of indices of the rows with positive diagonal entries. Using the inequality,  $q^{m-1}-1,\, q^{m-1}|\mathbb{F}_q^{+}|\leq q^{m}/2$ for all $m\in [n]$ and \eqref{eq:q bin coeff}, we get  
\[
    |\mathcal{L}_{k_1,k_2,k_3}| \leq
    \frac{1}{2^{k_1+k_2}}\sum_{1\leq \ell_1<\cdots<\ell_{k_1+k_2}\leq n}
    q^{\ell_1+\cdots+\ell_{k_1+k_2}}=\frac{q^{\binom{k_1+k_2+1}{2}}}{2^{k_1+k_2}}\qbinom{n}{k_1+k_2}.
\]
    
    Consider the case when $q\equiv 3\pmod{4}$. Observe that any two matrices in $\mathcal{L}_{k_1,k_2,k_3}$ that differ by a factor of $-1$ in one or more of their nonzero rows with zero diagonal entries give rise to the same $\PSD_2$ matrix. Therefore, 
\[
\begin{aligned}
        \psd_2(n) &\leq \sum_{k_1+k_2+k_3= n}\frac{1}{2^{k_1}}\left|\mathcal{L}_{k_1,k_2,k_3}\right|\\
        &\leq \sum_{k_1+k_2+k_3=n}\frac{q^{\binom{k_1+k_2+1}{2}}}{2^{2k_1+k_2}}\qbinom{n}{k_1+k_2}
        =\sum_{k=0}^n\frac{q^{\binom{k+1}{2}}}{2^{k}}\qbinom{n}{k}\left(\sum_{k_1=0}^{k}\frac{1}{2^{k_1}}\right)\\
        &\leq 2\sum_{k=0}^n\frac{q^{\binom{k+1}{2}}}{2^{k}}\qbinom{n}{k}.
    \end{aligned}
\]
    Taking $x=1/2$ in \eqref{eq:transf binom thm}, we obtain
\[
    \psd_2(n)\leq 2\prod_{i=1}^n\left(1+\frac{q^i}{2}\right).
\]
    This proves the bound for $q\equiv 3\pmod{4}$.

    Now, consider the case when $q\equiv 1\pmod{4}$. As in the previous case, any two matrices in $\mathcal{L}_{k_1,k_2,k_3}$ that differ by sign changes in the nonzero rows with zero diagonal entries give rise to the same $\PSD_2$ matrix. Moreover, if two matrices differ by sign changes in rows with nonzero diagonal entries, they also give rise to the same $\PSD_2$ matrix. Therefore, 
\[
    \begin{aligned}
    \psd_2(n)&\leq\sum_{k_1+k_2+k_3= n}\frac{1}{2^{k_1+k_2}}\left|\mathcal{L}_{k_1,k_2,k_3}\right|\\
    &\leq \sum_{k_1+k_2+k_3= n}\frac{q^{\binom{k_1+k_2+1}{2}}}{2^{2(k_1+k_2)}}\qbinom{n}{k_1+k_2}
    =\sum_{k=0}^n\frac{q^{\binom{k+1}{2}}}{2^{2k}}\qbinom{n}{k}.
    \end{aligned}
\]
    Taking $x=1/4$ in \eqref{eq:transf binom thm}, we obtain
\[
    \psd_2(n)\leq \prod_{i=1}^n\left(1+\frac{q^i}{4}\right).
\]
    This proves the bound for $q\equiv 1\pmod{4}$.
    
    Lastly, suppose that $q$ is even. In this case, the stated bound is simply the number of lower triangular matrices with nonnegative diagonal entries.
\end{proof}

We observed in \cref{sec:posdef 2} that $\PD_2(n)\subseteq\PD_1(n)$ for all $n\geq1$, with equality if and only if the field is definite. In the next theorem, we prove the reverse inclusion in the positive semidefinite case when the field is definite.

\begin{theorem}\label{thm:psd 1 psd 2 rel}
Over the definite field $\mathbb{F}_q$, $\PSD_1(n)\subseteq\PSD_2(n)$ for all $n\geq 1$.
\end{theorem}

\begin{proof}
We will use induction on the order $n$ of the matrices to prove that if $A\in \PSD_1(n)$, then $A\in \PSD_2(n)$. For $n=1$, the statement is trivially true. Suppose the statement is true when the order of the matrices is $n-1$.

Let $A\in \PSD_1(n)$. Since $A_{[n-1]}$ is also a $\PSD_1$ matrix, we can write $A_{[n-1]}=LL^T$ where $L\in M_{n-1}(\mathbb{F}_q)$ is a lower triangular matrix with nonnegative diagonal entry. Write 
\[
B=
\begin{pmatrix}
\begin{array}{c|c}
LL^T&v\\
\hline
\\[-1em]
v^T&a
\end{array}
\end{pmatrix},
\]
where $v\in \mathbb{F}_q^{n-1}$ and $a\in \mathbb{F}_q$. By the proof of \cref{thm:psd 1 recur}, $v^T$ linearly depends on the rows of $LL^T$. Therefore, $v^T$ can be written as $v^T=u^TLL^T$ for some $u\in \mathbb{F}_q^{n-1}$. Moreover, by the proof of \cref{thm:psd 1 recur}, $a=uLL^Tu^T+b$ for some nonnegative $b$. Since $\mathbb{F}_q$ is definite, there exists nonnegative $c$ such that $c^2=b$. This give us the decomposition 
\[
A=\begin{pmatrix}
\begin{array}{c|c}
L & \mathbf{0}_{(n-1)\times 1} \\ \hline
uL & c
\end{array}
\end{pmatrix}
\begin{pmatrix}
\begin{array}{c|c}
L^T & L^Tu^t \\ \hline
\mathbf{0}_{1\times (n-1)} & c
\end{array}
\end{pmatrix}
\]
Therefore $A$ is a $\PSD_2$ matrix, proving the theorem.
\end{proof}

\subsection{Type 3}
\label{sec:psd 3}

\begin{definition}
A matrix $A\in M_n(\mathbb{F}_q)$ is said to be \textit{positive semidefinite matrix of type 3}, denoted $\PSD_3$, if there exists $M\in M_n(\mathbb{F}_q)$ such that $A=MM^T$. We denote the set of $\PSD_3$ matrices in $M_n(\mathbb{F}_q)$ by $\PSD_3(n)$, and let $\psd_3(n,r)$ denote the number of rank $r$ matrices in $\PSD_3(n)$.
\end{definition} 

\begin{remark}
As in \cref{rem:posdef_2 and posdef_3}, we have $\PSD_2(n)\subseteq\PSD_3(n)$ for all $n\geq 1$.
\end{remark}

We will make use of the following classification of symmetric matrices over finite fields of odd characteristic.

\begin{proposition}[\cite{Albert1938}]\label{prop:char_of_symm_q_odd}
    Let $\mathbb{F}_q$ be a finite field of odd characteristic. Then every symmetric matrix in $A\in M_n(\mathbb{F}_q)$ of rank $r$ is congruent to either $\operatorname{diag}(\underbrace{1, \dots, 1}_{r },\underbrace{0, \dots, 0}_{n-r })$ or $\operatorname{diag}(\underbrace{1, \dots, 1}_{r-1},a,\underbrace{0, \dots, 0}_{n-r})$
    where $a\in \mathbb{F}_q^-$.
\end{proposition}

Recall that $N(n,r)$ represents the number of symmetric matrices of rank $r$ in $M_n(\mathbb{F}_q)$, given in \cref{thm:symm rank conunt}. 
We are now ready to give the formula for the number of $\PSD_3$ matrices of given rank in $M_n(\mathbb{F}_q)$. 

\begin{theorem}\label{thm:psd 3 count}
The number of matrices in $\PSD_3(n)$ of rank $r$ is given by 

\[
\psd_3(n,r) = 
\begin{cases}

\displaystyle
\prod_{i=1}^{\lfloor\frac{r-1}{2}\rfloor}\dfrac{q^{2i}}{q^{2i}-1}\prod_{i=0}^{r-1}(q^{n-i}-1)
& \text{if $q$ is even}, \\[1em]

\displaystyle
N(n,r) 
& \text{if $r< n$ and  $q$ is odd}\\[1em]

\displaystyle
\pd_3(n)
& \text{if $r=n$ and $q$ is odd}.
\end{cases}
\]

\end{theorem}

\begin{proof}
For the proof of the first and third cases, reader can refer to \cite{Macwilliams1969}. We will prove the case when $r<n$ and $q$ is odd. 
As mentioned in \cref{prop:char_of_symm_q_odd}, every symmetric matrix of rank $r$ is congruent to 
$\operatorname{diag}(\underbrace{1, \dots, 1}_{r },\underbrace{0, \dots, 0}_{n-r })$ or $\operatorname{diag}(\underbrace{1, \dots, 1}_{r-1},a,\underbrace{0, \dots, 0}_{n-r})$, where $a$ is a negative element in $\mathbb{F}_q$. 
In the first case, it trivially follows that $A=BB^T$ for some $B\in M_n(\mathbb{F}_q)$. Suppose now that $A$ lies in the second case. By \cref{prop:php fin fld}(1), every element in $\mathbb{F}_q$ can be written as a sum of two squares. Let $a=b^2+c^2$ for some $b,c\in\mathbb{F}_q$. Now consider the matrix $B\in M_n(\mathbb{F}_q)$ in which the first $r-1$ diagonal entries are $1$, the $(r,r)$'th entry is $b$, the $(r,r+1)$'th entry is $c$, and all other entries are $0$. That is,
\[
B =
\left(
\begin{array}{c|c|c}
\iden_{r-1} & \textbf{0}_{(r-1)\times2} & \textbf{0}_{(r-1)\times (n-r-1)} \\ \hline
\textbf{0}_{2\times (r-1)} & \begin{array}{cc} b & c \\ 0 & 0 \end{array} & \textbf{0}_{2\times (n-r-1)} \\ \hline
\textbf{0}_{(n-r-1)\times (r-1)} & \textbf{0}_{(n-r-1)\times 2} & \textbf{0}_{(n-r-1)\times (n-r-1)}
\end{array}
\right).
\]
Observe that $BB^T=\operatorname{diag}(\underbrace{1, \dots, 1}_{r-1 \text{ times}},a,\underbrace{0, \dots, 0}_{n-r \text{ times}})$. Therefore, $A=(MB)(MB)^T.$
\end{proof}

\subsection{Type 4}
\label{sec:psd 4}

\begin{definition}
    A symmetric matrix in $M_n(\mathbb{F}_q)$ is said to be \textit{positive semidefinite of type 4} (denoted $\PSD_4$) if it is orthogonally similar to a diagonal matrix with all of its diagonal entry nonnegative. We denote the set of $\PSD_4$ matrices in $M_n(\mathbb{F}_q)$ by $\PSD_4(n)$, and let $\psd_4(n,r)$ denote the number of rank $r$ matrices in $\PSD_4(n)$.
\end{definition}

\begin{remark}
By an argument similar to that in \cref{rem:posdef 4 and posdef 3 rel}, we have $\PSD_4(n)\subseteq\PSD_3(n)$ for all $n\geq 1$.
\end{remark}

The number of orthogonal matrices in $\mathrm{GL}_n(\mathbb{F}_q)$ is given in \cref{thm:orth matrices count}. The number of matrices in $\PSD_4(n)$ admits a formula analogous to that for $\pd_4(n)$, and satisfies a similar generating function.

\begin{theorem}\label{thm:psd 4 enum}
The number of matrices in $\PSD_4(n)$ of rank $r$ is given by
\[
\psd_4(n,r)=\sum_{\substack{k_1 + \cdots + k_{s} = n-r \\ k_i \geq 0}} 
\frac{|\mathrm{O}_n(\mathbb{F}_q)|}{|\mathrm{O}_{k_1}(\mathbb{F}_q)||\mathrm{O}_{k_2}(\mathbb{F}_q)|\cdots|\mathrm{O}_{k_s}(\mathbb{F}_q)||\mathrm{O}_{r}(\mathbb{F}_q)|},
\]
where $s=|\mathbb{F}_q^{+}|$. 
\end{theorem}

\begin{proof}
The proof follows the same reasoning as in \cref{thm:posdef 4 count}, except that in this case, zero is also allowed as a diagonal entry with multiplicity $r$.
\end{proof}

\begin{corollary}
    Let $\psd_4(n)$ denote the number of positive definite matrices of type 4 in $M_n(\mathbb{F}_q)$. The sequence $(\psd_4(n))_{n\geq 0}$ satisfies the generating function-type equation, 
\[
    \sum_{n\geq0}\psd_4(n)\frac{x^n}{|\mathrm{O}_n(\mathbb{F}_q)|}=\left(\sum_{k\geq 0}\frac{x^k}{|\mathrm{O}_k(\mathbb{F}_q)|}\right)^{s+1},
\]
    where $s=|\mathbb{F}_q^{+}|$.
\end{corollary}

\subsection{Type 5}
\label{sec:psd 5}

\begin{definition}
A symmetric matrix in $M_n(\mathbb{F}_q)$ is said to be \textit{positive semidefinite of type 5} (denoted $\PSD_5$), if all it principal minors are nonnegative. 
We denote the set of $\PSD_5$ matrices in $M_n(\mathbb{F}_q)$ by $\PSD_5(n)$, and let $\psd_5(n)=|\PSD_5(n)|$.
\end{definition}

\begin{remark}\label{rem:psd 5 even}
    Since every element of a field $\mathbb{F}_q$ of characteristic two is nonnegative, every symmetric matrix in $M_n(\mathbb{F}_q)$ is a $\PSD_5$ matrix.
\end{remark}

By \cref{rem:psd 5 even} the only interesting case remaining is when $\mathbb{F}_q$ has odd characteristic. We start by giving formula for the number of matrices in $\PSD_5(2)$.

\begin{proposition}\label{prop: psd 5 in M2(Fq)}
    Let $\mathbb{F}_q$ be a field of odd characteristic. The number of matrices in $\PSD_5(2)$ is given by 
\[
\psd_5(2)=\begin{cases}
    \displaystyle
    3\left(\frac{q-1}{2}\right)^2+(q-1)\left(\frac{q-5}{4}\right)\left(\frac{q-1}{2}\right)+q(2q-1)
    & \text{if $q\equiv 1 \pmod 4$}, \\[1em]

    \displaystyle
    3\left(\frac{q-1}{2}\right)^2+(q-1)\left(\frac{q-3}{4}\right)\left(\frac{q-1}{2}\right)+(2q-1)
    & \text{if $q\equiv 3 \pmod 4$}.
    \end{cases}
\]
\end{proposition}

\begin{proof}
For
\[
A=\begin{pmatrix}
a & c \\
c & b
\end{pmatrix}
\]
to be a $\PSD_5$ matrix, $a$, $b$, and $\det(A)$ must be nonnegative. We partition the set of $\PSD_5$ matrices in $M_2(\mathbb{F}_q)$ into three cases and analyze each separately.

\noindent \textbf{Case 1:} $a,b \in \mathbb{F}_q^{+}$ and $c=0$. In this case, there are $((q-1)/2)^2$ choices, and all are valid.

\noindent \textbf{Case 2:} $a,b \in \mathbb{F}_q^{+}$ and $c \neq 0$. In this case, $\det(A)=ab-c^2 = c^2(ab/c^2-1)$ must be nonnegative.
First, consider the case $\det(A)=0$. For each choice of $a,b \in \mathbb{F}_q^{+}$, there are exactly two choices of $c \in \mathbb{F}_q^\times$, that is $x$ and $-x$, where $x^2=ab$. Hence, there are $2((q-1)/2)^2$
such matrices.
Now suppose $\det(A) > 0$. By \cref{thm:a a-1 is +ve}, there are $(q-5)/4$ choices for $ab/c^2$ when $q \equiv 1 \pmod{4}$ and $(q-3)/4$ choices for $ab/c^2$ when $q \equiv 3 \pmod{4}$. For each such choice, there are $q-1$ choices for $c$ and $(q-1)/2$ choices for the pair $(a,b)$. Therefore, this contributes
\[
(q-1)\left(\frac{q-5}{4}\right)\left(\frac{q-1}{2}\right)
\]
matrices when $q \equiv 1 \pmod{4}$, and
\[
(q-1)\left(\frac{q-3}{4}\right)\left(\frac{q-1}{2}\right)
\]
matrices when $q \equiv 3 \pmod{4}$.

\noindent \textbf{Case 3:} At least one of $a$ or $b$ is zero. In this case, there are $2q-1$ ways to choose the pair $(a,b)$.
If $q \equiv 1 \pmod{4}$, then for any such pair $(a,b)$, every $c \in \mathbb{F}_q$ yields a $\PSD_5$ matrix, giving $q$ choices for $c$.
If $q \equiv 3 \pmod{4}$, then for any such pair $(a,b)$, only $c=0$ yields a $\PSD_5$ matrix, giving exactly one choice for $c$.
\end{proof}

We next establish an inclusion relation between the classes $\PSD_5(n)$ and $\PSD_1(n)$ over the finite field $\mathbb{F}_q$, where $q\equiv3\pmod4$ is a prime power.

\begin{proposition}\label{prop:psd 1 psd 5 rel}
Let $q\equiv3\pmod4$, where $q$ is a prime power. Over the field $\mathbb{F}_q$, we have $\PSD_5(n)\subseteq\PSD_1(n)$ for every $n\ge1$.
\end{proposition}

\begin{proof}
We prove the statement by induction on $n$. For $n=1$, the statement is
trivially true. Assume that the statement holds for matrices of order
$n-1$.

Let $B\in \PSD_5(n)$ and write
\[
B=
\begin{pmatrix}
\begin{array}{c|c}
A&v\\
\hline
\\[-1em]
v^T&a_2
\end{array}
\end{pmatrix},
\]
where $A\in\PSD_5(n-1)$, $v\in \mathbb{F}_q^{n-1}$, and
$a_2\in\mathbb{F}_q$.

By the induction hypothesis, $A\in\PSD_1(n-1)$. If $v^T$ is linearly dependent on the rows of $A$, then it must satisfy either Case~1 or Case~2 of \cref{thm:psd 1 recur} for every principal minor of $B$ to be positive. Therefore, $B\in\PSD_1(n)$.

Suppose that $v$ is linearly independent of the rows of $A$. In this case, $A$ cannot have full rank, since otherwise the row space of
$A$ would contain $v^T$. First suppose that $\operatorname{rank}(A)=n-2$.
By \cref{prop:char_of_symm_q_odd}, there exists
$C\in \mathrm{GL}_{n-1}(\mathbb{F}_q)$ such that
\[
\operatorname{diag}(\underbrace{1,\dots,1}_{n-2},0)=C^TAC.
\]
Therefore,
\[
\left(
\begin{array}{c|c|c}
\iden_{n-2}&\mathbf{0}_{(n-2)\times1}&u\\
\hline
\mathbf{0}_{1\times(n-2)}&0&y\\
\hline
\\[-1em]
u^T&y&a_2
\end{array}
\right)
=
\left(
\begin{array}{c|c}
C^T&\mathbf{0}_{(n-1)\times1}\\
\hline
\mathbf{0}_{1\times(n-1)}&1
\end{array}
\right)
B
\left(
\begin{array}{c|c}
C&\mathbf{0}_{(n-1)\times1}\\
\hline
\mathbf{0}_{1\times(n-1)}&1
\end{array}
\right),
\]
where $v^TC=\begin{pmatrix}u^T&y\end{pmatrix}$ and, by \cref{lem:mac symm class}(3), $y\neq0$. Expanding the determinant of the matrix on the left along the last row, we
obtain
\[
\det
\left(
\begin{array}{c|c|c}
\iden_{n-2}&\mathbf{0}_{(n-2)\times1}&X^T\\
\hline
\mathbf{0}_{1\times(n-2)}&0&y\\
\hline
X&y&a_2
\end{array}
\right)
=-y^2.
\]
Therefore,
\[
\det(B)=-\det(C)^{-2}y^2,
\]
which is negative. Hence $B\notin\PSD_5(n)$.

Now consider the case when $\operatorname{rank}(A)=r\le n-3$. Let
$I\subseteq[n-1]$ be an index set of size $r$ such that $A_I$ has rank
$r$. If $v_I^T$ linearly depends on the rows of $A_I$, then
there exists some $j_1\in[n-1]\setminus I$ such that
$v_{I\cup\{j_1\}}^T$ is linearly independent of the rows of
$A_{I\cup\{j_1\}}$. Applying the above argument to the principal submatrix
$A_{I\cup\{j_1,n\}}$, we obtain $\det(A_{I\cup\{j_1,n\}})$ to be negative. Similarly, if $v_{I}^T$ is linearly independent of the rows of
$A_I$, choose any $j_2\in[n-1]\setminus I$. Then the same argument applied
to $A_{I\cup\{j_2,n\}}$ gives $\det(A_{I\cup\{j_2,n\}})$ is negative.

Therefore, we have shown that $v^T$ cannot be linearly independent of the rows of $A$. Hence, $v^T$ linearly depends on the rows of $A$, in which case we have proved that $B$ belongs to $\PSD_1(n)$. Therefore, by the induction, the proposition follows.
\end{proof}

At present, we do not have a simple closed-form expression for the number of matrices in $\PSD_5(n)$ over $\mathbb{F}_q$ when $n\geq 3$ and $q$ is odd. In the absence of such a formula, we instead obtain a structural formula for number of matrices in $\PSD_5(n)$ is given in \cref{thm:psd 5 struc formula}.

To prove \cref{thm:psd 5 struc formula}, we use tools from arithmetic geometry, namely the theory of Weil zeta functions. We begin by recalling the definition of an algebraic set.

\begin{definition}
Let $\{f_1,f_2,\dots,f_k\} \subseteq \mathbb{F}_q[x_1,x_2,\dots,x_n]$ be a system of polynomial equations. The \textit{algebraic set} defined by the system $\{f_1,f_2,\dots,f_k\}$ is the set of all common zeros of these polynomials in $\mathbb{F}_q^n$, that is,
\[
\{(a_1,a_2,\dots,a_n)\in \mathbb{F}_q^n \mid f_i(a_1,a_2,\dots,a_n)=0 \text{ for all } i\in [k]\}.
\]
For each integer $m \geq 1$, let $N_m(f_1,f_2,\dots,f_k)$ denote the number of points in the algebraic set defined by $\{f_1,f_2,\dots,f_k\}$ over the affine space $\mathbb{F}_{q^m}^n$, that is,
\[N_m(f_1,f_2,\dots,f_k)=\left|\left\{
(a_1,a_2,\dots,a_n)\in \mathbb{F}_{q^m}^n
\;\middle|\;
f_i(a_1,a_2,\dots,a_n)=0 \text{ for all } i\in [k]
\right\}\right|.
\]
\end{definition}

We describe the set $\PSD_5(n)$ as a disjoint union of subsets that can each be expressed as the projection of an algebraic set. 
Let $p$ be an odd prime, let $n$ be a positive integer, and let $\mathcal{S}=(S_1,S_2,\dots,S_{2^n-1})$ be an ordering of all nonempty subsets of $[n]$. 
For every nonempty subset $S \subseteq [n]$, define $\mathcal{X}_{S,m}(0)$ and $\mathcal{X}_{S,m}(1)$ to be the sets of all matrices $A \in M_n(\mathbb{F}_{p^m})$ such that $\det(A_S)=0$ and $\det(A_S)\in\mathbb{F}_{p^m}^{+}$, respectively.

For $\epsilon = (\epsilon_1, \dots, \epsilon_{2^n-1}) \in \{0,1\}^{2^n-1}$, we define
\[
\mathcal{X}_m(\epsilon)=\bigcap_{i\in [2^{n}-1]} \mathcal{X}_{S_i,m}(\epsilon_{i}).
\]
Observe that for each $\epsilon \in \{0,1\}^{2^n-1}$, $\mathcal{X}_m(\epsilon)$ is a subset of $\PSD_5(n)$ over $\mathbb{F}_{p^m}$.

Conversely, let $A \in \PSD_5(n)$ over $\mathbb{F}_{p^m}$. For each $i \in [2^{n}-1]$, define $\epsilon_{i} = 0$ if $\det(A_{S_i}) = 0$, and $\epsilon_{i} = 1$ if $\det(A_{S_i}) \in\mathbb{F}_{p^m}^+$. Let $\epsilon_A = (\epsilon_1, \epsilon_2, \dots, \epsilon_{2^n-1})$. Then $A \in \mathcal{X}_m(\epsilon_A)$. Moreover, if $\epsilon \in \{0,1\}^{2^n-1}$ is such that $A \in \mathcal{X}_m(\epsilon)$, then necessarily $\epsilon = \epsilon_A$. 

Therefore, the set $\PSD_5(n)$ over $\mathbb{F}_{p^m}$ can be written as the disjoint union
\begin{equation}\label{eq:psd 5 disjoint union}
\PSD_5(n)=\bigsqcup_{\epsilon \in \{0,1\}^{2^n-1}} \mathcal{X}_m(\epsilon).
\end{equation}

Fix $\epsilon \in \{0,1\}^{2^n-1}$. We now express $\mathcal{X}_m(\epsilon)$ as the projection of an algebraic set defined by a suitable collection of polynomials. Let $T \subseteq [2^n-1]$ denote the set of indices $j$ for which $\epsilon_j = 1$, and write $t = |T|$. Let $X = (x_{a,b})_{1 \leq a,b \leq n}$ be an $n \times n$ matrix of variables, and let $Y =\{y_i\mid i\in T\}$ be a set of additional variables. Consider the following system of polynomial equations in $\mathbb{F}_{p}[X,Y]$:

\begin{itemize}
\item For each $i \in T$,
\[g_i(X,Y)=y_i^2 \det\big(X_{S_i}\big) - 1;
\]

\item For each $i\in [2^n-1]\setminus T$ ,
\[g_i(X,Y) = \det\big({X_{S_i}}\big).
\]
\end{itemize}
and for each positive integer $k$, let $\mathcal{X}_m^0(\epsilon) \subseteq \mathbb{F}_{p^m}^{n^2+t}$ to be the algebraic set defined by the above system of polynomial equations.

The sets $\mathcal{X}_m(\epsilon)$ and $\mathcal{X}_m^0(\epsilon)$ satisfy the two properties given in the following lemmas.

\begin{lemma}\label{lem:part psd 5 1}
If $(X',Y') \in \mathcal{X}_m^0(\epsilon)$, then $X' \in \mathcal{X}_m(\epsilon)$. Conversely, if $X' \in \mathcal{X}_m(\epsilon)$, then there exists $Y'=(y_i)_{i \in T} \in \mathbb{F}_{p^m}^{t}$ such that $(X',Y') \in \mathcal{X}_m^0(\epsilon)$.
\end{lemma}

\begin{proof}
For the first part, observe that for each $i \in T$, the equation $g_i(X',Y') = 0$ implies that
\[
\det\big(X'_{S_{i}}\big) = y_i^{-2} \in \mathbb{F}_{p^m}^+.
\]
Further, for each $i \in [2^n-1] \setminus T$, we have 
\[
\det\big(X'_{S_{i}}\big) = 0
\] 
by definition of $\mathcal{X}_m^0(\epsilon)$. Therefore, $X'\in \mathcal{X}_m(\epsilon)$.

For the converse, suppose that $X' \in \mathcal{X}_m(\epsilon)$. Then for each $i \in T$, we have 
\[
\det\big(X'_{S_{i}}\big) \in \mathbb{F}_{p^m}^+.
\] 
Let $Y'=(y_i)_{i \in T}$ where $y_i \in \mathbb{F}_{p^m}$ satisfies the equation
\[y_i^2 = \big(\det\big(X'_{S_{i}}\big)\big)^{-1}.
\]
With this choice of $Y'$, we have $(X',Y')\in \mathcal{X}_m^0(\epsilon)$.
Thus $(X',Y') \in \mathcal{X}_m^0(\epsilon)$, completing the proof.
\end{proof}

For a given $X'\in \mathcal{X}_m(\epsilon)$, we now want to classify all $Y'\in \mathbb{F}_{p^m}^{t}$, such that $(X',Y')\in \mathcal{X}_m^0(\epsilon)$. Let $\sigma=(\sigma_i)_{i\in T}\in \{-1,1\}^{t}$. For $Y=(y_i)_{i\in T} \in \mathbb{F}_q^{t}$, define
\[
\sigma  Y := (\sigma_i y_i)_{i \in T} \in \mathbb{F}_q^{t}. 
\]

\begin{lemma} \label{lem:part psd 5 2}
If $(X',Y') \in \mathcal{X}_m^0(\epsilon)$, then $(X', \sigma Y') \in \mathcal{X}_m^0(\epsilon)$ for every $\sigma \in \{-1,1\}^{t}$. Conversely, if $(X',Y'), (X',Y_0') \in \mathcal{X}_m^0(\epsilon)$, then there exists $\sigma \in \{-1,1\}^{t}$ such that $Y' = \sigma Y_0'$.
\end{lemma}

\begin{proof}
Suppose $(X',Y') \in \mathcal{X}_m^0(\epsilon)$. Let $\sigma=(\sigma_i)_{i\in T}\in \{-1,1\}^{t}$. Then for each $i\in T$,
\[ (\sigma_i y_i')^2 \det\big(X'_{S_{i}}\big)-1
= (y_i')^2 \det\big(X'_{S_{i}}\big)-1
= 0, 
\]
Hence $(X', \sigma  Y') \in \mathcal{X}_m^0(\epsilon)$.

For the converse, suppose $(X',Y'), (X',Y_0') \in \mathcal{X}_m^0(\epsilon)$. Then for each $i \in T$,
\[(y_i')^2 \det\big(X'_{S_{i}}\big)-1= (y_{0,i}')^2 \det\big(X'_{S_{i}}\big) -1.
\]
Thus,
\[(y_i')^2 = (y_{0,i}')^2,
\]
which implies that $y_i'= \pm y_{0,i}'$. Therefore, there exists $\sigma_i \in \{-1,1\}$ such that $y_i'= \sigma_i y_{0,i}'$ for each $i\in T$. Let $\sigma = (\sigma_i)_{i \in T}$. Then $Y' = \sigma Y_0'$, completing the proof.
\end{proof} 

What we have proved so far is that $\PSD_5(n)$ over $\mathbb{F}_{p^m}$ can be expressed as a disjoint union of subsets, each of which can be realized as the projection of an algebraic set. Thus, for $f_1,f_2,\dots,f_k \in \mathbb{F}_p[x_1,x_2,\dots,x_n]$, if we can say something about $N_m(f_1,f_2,\dots,f_k)$, that is, the number of points in the corresponding algebraic set, then we can obtain information about $\mathcal{X}_m(\epsilon)$ for each $\epsilon \in \{0,1\}^{2^n-1}$. We now introduce the Weil zeta function to tackle this problem.

\begin{definition}
Let $\{f_1,f_2,\dots,f_k\} \subseteq \mathbb{F}_q[x_1,x_2,\dots,x_n]$ be system of polynomial equation. The \textit{Weil zeta function} $\zeta_{f_1,f_2,\dots,f_k}^q(z)$ is defined by
\[
\zeta_{f_1,f_2,\dots,f_k}^q(z)
=
\exp\left(
\sum_{m\geq 1}
N_m(f_1,f_2,\dots,f_k)\frac{z^m}{m}
\right).
\]
\end{definition}

In 1949, Weil conjectured four properties of the Weil zeta function: rationality, a functional equation (Poincaré duality), an analogue of the Riemann Hypothesis, and a connection between Betti numbers and topology. All of these properties are now known to be true.
We will use only the rationality of the Weil zeta function due to Dwork.

\begin{theorem}[{\cite{Dwork1960}}]
\label{thm: dwork thm}
    The Weil zeta function is a rational function over $\mathbb{Q}$, where the constant term of both the numerator and denominator is 1.
\end{theorem}

\begin{lemma}\label{lem:rational funcn}
Every rational function $f(z)=P(z)/Q(z) \in \mathbb{Q}(z)$ with $P(0), Q(0) \neq 0$ can be written in the form
\[
f(z)=\frac{P(0)}{Q(0)} \exp\left(\sum_{m \geq 1} \left( \sum_{j=1}^{b} \beta_j^m - \sum_{i=1}^{a} \alpha_i^m \right)\frac{z^m}{m} \right),
\]
where $\alpha_1, \dots, \alpha_a$ are the reciprocals of the roots of $P(z)$ and $\beta_1, \dots, \beta_b$ are the reciprocals of the roots of $Q(z)$. In particular, for $f_1, f_2, \dots, f_k \in \mathbb{F}_q[x_1, x_2, \dots, x_n]$, there exists algebraic numbers $\gamma_1, \dots, \gamma_{n_1}$ and $\delta_1, \dots, \delta_{n_2}$ such that
    \[N_m(f_1, f_2, \dots, f_k) = \sum_{j=1}^{n_2} \delta_j^m - \sum_{i=1}^{n_1} \gamma_i^m.
\]
\end{lemma}

\begin{proof}
    We can write $P(z)=P(0)\prod_{i=1}^{a}(1-\alpha_iT)$ and $Q(z)=Q(0)\prod_{i=1}^{b}(1-\beta_iT)$. This gives 
\[
\begin{aligned}
f(z) &=\frac{P(0)}{Q(0)}\frac{\prod_{i=1}^{a}(1-\alpha_iz)}{\prod_{i=1}^{b}(1-\beta_iz)}\\
&=\frac{P(0)}{Q(0)}\exp{\left(\sum_{i=1}^{a}\log(1-\alpha_iz)-\sum_{j=1}^{b}\log(1-\beta_jz)  \right)}\\  
&=\frac{P(0)}{Q(0)}\exp{\left(\sum_{m\geq 1}\left(\sum_{j=1}^{b}\beta_j^m -\sum_{i=1}^{a}\alpha_i^m \right)\frac{z^m}{m}\right)}
\end{aligned}
\]
This proves the first part of the lemma. 

For the second part, we use \cref{thm: dwork thm} together with the first part of the lemma. Comparing the coefficients, we obtain the desired form.
\end{proof}

We are now ready to state the structural formula for the number of matrices in $\PSD_5(n)$.

\begin{theorem}\label{thm:psd 5 struc formula}
Let $p$ be an odd prime and let $n$ be a positive integer. Then there exists a positive integer $r$, integers $c_{1,p}, c_{2,p}, \dots, c_{r,p}$, and algebraic numbers $\gamma_{1,p}, \gamma_{2,p}, \dots, \gamma_{r,p}$ such that the number of matrices in $\PSD_5(n)$ over $\mathbb{F}_{p^m}$ is given by

\[2^{-(2^n-1)}\left(c_{1,p} \gamma_{1,p}^{m} + c_{2,p} \gamma_{2,p}^{m} + \cdots + c_{r,p} \gamma_{r,p}^{m}\right).
\]
\end{theorem}

\begin{proof} 
From \eqref{eq:psd 5 disjoint union}, it follows that the number of matrices in $\PSD_5(n)$ over $\mathbb{F}_{p^m}$ is given by
\begin{equation}\label{eq:psd 5 xm count}
\psd_5(n)=\sum_{\epsilon \in \{0,1\}^{2^n-1}} 
\big|\mathcal{X}_m(\epsilon)\big|.
\end{equation}
Therefore, for each $\epsilon \in \{0,1\}^{2^n-1}$, we aim to compute $\big|\mathcal{X}_m(\epsilon)\big|$. 

Fix $\epsilon \in \{0,1\}^{2^n-1}$. We define an equivalence relation $\sim$ on $\mathcal{X}_m^0(\epsilon)$ where we define $(X',Y') \sim (X_0',Y_0')$ if $X' = X_0'$. By \cref{lem:part psd 5 1}, each equivalence class corresponds to a unique element of $\mathcal{X}_m(\epsilon)$, and by \cref{lem:part psd 5 2}, each equivalence class has cardinality $2^t$. Hence,
\begin{equation}\label{eq:psd 5 xm x0m rel}
|\mathcal{X}_m^0(\epsilon)| = 2^{t}\, \big|\mathcal{X}_m(\epsilon)\big|.
\end{equation}
Applying \cref{lem:rational funcn} to the polynomials defining the algebraic set $\mathcal{X}_m^0(\epsilon)$, namely
\begin{itemize}
\item for each $i \in T$,
\[g_i(X,Y) = y_i^2 \det\big(X_{S_{i}}\big) - 1,
\]
\item for each $i \in [2^n-1] \setminus T$,
\[g_i(X,Y) = \det\big(X_{S_{i}}\big),
\]
\end{itemize}
we obtain algebraic numbers $\alpha_1, \dots, \alpha_{n_1}$ and $\beta_1, \dots, \beta_{n_2}$ such that
\[
|\mathcal{X}_m^0(\epsilon)| = \sum_{j=1}^{n_2} \beta_j^m - \sum_{i=1}^{n_1} \alpha_i^m.
\]
Substituting into \eqref{eq:psd 5 xm x0m rel}, we obtain
\[
\big| \mathcal{X}_m(\epsilon) \big|=2^{-t}
\left(\sum_{j=1}^{n_2} \beta_j^m-\sum_{i=1}^{n_1} \alpha_i^m\right).
\]

Since this holds for every $\epsilon \in \{0,1\}^{2^n-1}$, the result follows from \eqref{eq:psd 5 xm count}, completing the proof of the theorem.
\end{proof}

\begin{remark}\label{rem:when psd 5 polyn}
Suppose $\gamma_{1,p}, \gamma_{2,p}, \dots, \gamma_{r,p}$ in \cref{thm:psd 5 struc formula} are of the form $\gamma_{i,p}=\zeta_i p^{\ell_i}$,
where $\zeta_i$'s are some roots of unity, and $\ell_i\in\mathbb{Z}$ for all $i\in[t]$. Let $t$ be the lcm of the orders of $\zeta_1,\dots,\zeta_r$. 
Then, for each $b \in [0,t-1]$, the number of matrices in $\PSD_5(n)$ over $\mathbb{F}_{p^{at+b}}$ is a polynomial in $p^a$.
\end{remark}

\cref{rem:when psd 5 polyn} and \cref{prop: psd 5 in M2(Fq)} motivate us to conjecture the following. 

\begin{conjecture}
For a prime $p$ and positive integer $n$, there exists a positive integer $t$ and rational polynomials $f_0(x),\dots,f_{t-1}(x)\in \mathbb{Q}[x]$ of degree $\binom{n+1}{2}$, such that the number of matrices in $\PSD_5(n)$ over $\mathbb{F}_{p^m}$ is $f_i(p^k)$ when $k\equiv i\pmod{t}$.
\end{conjecture}

\subsection{Relating these definitions} 
\label{sec:rel psd}
We now examine how two these criteria are related. 
The Venn diagrams in \cref{fig:psd-venn} summarize the inclusion relations between the sets of positive semidefinite matrices of various types.
We justify these figures below.
We should clarify that we have not verified all the inclusions in these figures, and these should be interpreted as heuristic pictures, rather than as completely accurate. For example, in \cref{fig:psd-venn}(A), we do not know if there are matrices which are both $\PSD_1$ and $\PSD_5$ but neither $\PSD_2$ nor $\PSD_3$.

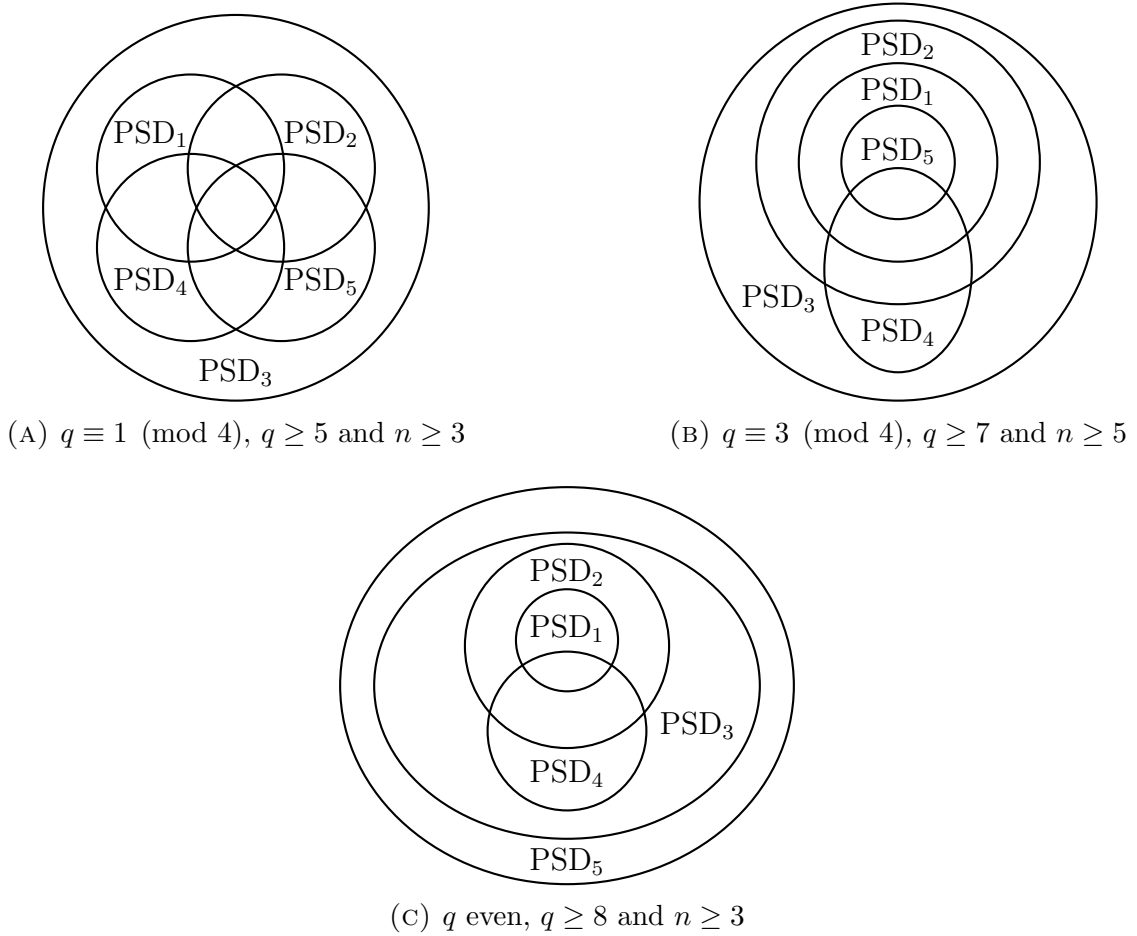
\begin{figure}[htbp]
\centering

\begin{subfigure}{0.45\textwidth}
\centering
\begin{tikzpicture}[scale=0.75]

\draw[thick] (0,0) circle (3.4);
\node at (0,-2.9) {$\PSD_{3}$};

\draw[thick] (-0.8,0.7) circle (1.65);
\node at (-1.5,1.3) {$\PSD_{1}$};

\draw[thick] (0.8,0.7) circle (1.65);
\node at (1.5,1.3) {$\PSD_{2}$};

\draw[thick] (-0.8,-0.7) circle (1.65);
\node at (-1.5,-1.3) {$\PSD_{4}$};

\draw[thick] (0.8,-0.7) circle (1.65);
\node at (1.5,-1.3) {$\PSD_{5}$};

\end{tikzpicture}
\caption{$q\equiv1\pmod4$, $q\ge5$ and $n\ge 3$}
\end{subfigure}
\hfill
\begin{subfigure}{0.45\textwidth}
\centering
\begin{tikzpicture}[scale=0.75]

\draw[thick] (0,0) circle (3.5);
\node at (-2.1,-1.7) {$\PSD_{3}$};

\draw[thick] (0,0.7) circle (1);
\node at (0,0.9) {$\PSD_{5}$};

\draw[thick] (0,0.7) circle (1.75);
\node at (0,1.95) {$\PSD_{1}$};

\draw[thick] (0,0.7) circle (2.5);
\node at (0,2.75) {$\PSD_{2}$};

\draw[thick] (0,-1.2) ellipse (1.3 and 1.8);
\node at (0,-2.3) {$\PSD_{4}$};

\end{tikzpicture}
\caption{$q\equiv3\pmod4$, $q\ge7$ and $n\ge 5$}
\end{subfigure}

\vspace{1em}

\begin{subfigure}{0.45\textwidth}
\centering
\begin{tikzpicture}[scale=0.75]

\draw[thick] (0,0) ellipse (4 and 3.5);
\node at (0,-3.1) {$\PSD_{5}$};

\draw[thick] (0,0) ellipse (3.4 and 2.7);
\node at (2.3,-0.7) {$\PSD_{3}$};

\draw[thick] (0,0.8) circle (0.9);
\node at (0,1) {$\PSD_{1}$};

\draw[thick] (0,0.7) circle (1.8);
\node at (0,2) {$\PSD_{2}$};

\draw[thick] (0,-0.8) circle (1.4);
\node at (0,-1.55) {$\PSD_{4}$};

\end{tikzpicture}
\caption{$q$ even, $q\ge8$ and $n\ge 3$}
\end{subfigure}

\caption{Relations between the sets of positive semidefinite matrices of various types over $\mathbb{F}_q$. }
\label{fig:psd-venn}

\end{figure}

\subsubsection{$\PSD_{1}(n)$ and $\PSD_{2}(n)$} 
From \cref{sec:posdef 2}, we know that $\PD_2(n)\subseteq \PD_1(n)$, where equality holds if and only if the field is definite. However, this is not true in the positive semidefinite case. In \cref{thm:psd 1 psd 2 rel}, we proved that $\PSD_1(n)\subseteq \PSD_2(n)$ when the field is definite. We now give an example of a matrix in $\PSD_2(n)\setminus \PSD_1(n)$ over an arbitrary finite field.

First, consider the case when $q$ is even. Consider the matrix 
\[
B_1=\left(
\begin{array}{c|c}
\begin{matrix}
0 & 0 & 0 \\
0 & 0 & 1 \\
0 & 1 & 0
\end{matrix}
&
\textbf{0}_{3\times (n-3)} \\
\hline
\textbf{0}_{(n-3)\times 3} & \iden_{n-3}
\end{array}
\right)=\left(
\begin{array}{c|c}
\begin{matrix}
0 & 0 & 0 \\
1 & 1 & 0 \\
0 & 1 & 1
\end{matrix}
&
\textbf{0}_{3\times (n-3)} \\
\hline
\textbf{0}_{(n-3)\times 3} & \iden_{n-3}
\end{array}
\right)\left(
\begin{array}{c|c}
\begin{matrix}
0 & 1 & 0 \\
0 & 1 & 1 \\
0 & 0 & 1
\end{matrix}
&
\textbf{0}_{3\times (n-3)} \\
\hline
\textbf{0}_{(n-3)\times 3} & \iden_{n-3}
\end{array}
\right)
\]
Observe that $B_1 \in \PSD_2(n)\setminus \PSD_1(n)$ because $(B_1)_{S}$ is not $PD_1$ for any $S \subseteq [3]$ and $B_1$ is not a zero matrix.
    
Now consider the case when $q$ is odd. It is easy to observe that we can choose $a\in \mathbb{F}_q^{+}$ such that $a+1\in \mathbb{F}_q^{-}$. Let $b^2=a$. Consider the matrix
\[
B_2=\left(
\begin{array}{c|c}
\begin{matrix}
0 & 0 \\
0 & a+1 
\end{matrix}
&
\textbf{0} \\
\hline
\textbf{0} & \iden_{n-2}
\end{array}
\right)=\left(
\begin{array}{c|c}
\begin{matrix}
0 & 0 \\
b & 1 
\end{matrix}
&
\textbf{0} \\
\hline
\textbf{0} & \iden_{n-2}
\end{array}
\right)\left(
\begin{array}{c|c}
\begin{matrix}
0 & b \\
0 & 1 
\end{matrix}
&
\textbf{0} \\
\hline
\textbf{0} & \iden_{n-2}
\end{array}
\right).
\] 
Again observe that $B_2 \in \PSD_2(n)\setminus \PSD_1(n)$. 
In \cref{sec:PD_1_PD_2_rel}, we saw examples of matrices in $\PD_1(n)\setminus \PD_2(n)$ when the field is indefinite, and therefore we obtain matrices in $\PSD_1(n)\setminus \PSD_2(n)$.

\subsubsection{$\PSD_{3}(n)$ with $\PSD_{1}(n)$ and $\PSD_{2}(n)$}
From \cref{sec:PD_3_PD_1 2 4_rel}, we have the inclusion relation $\PD_1(n),\,\PD_2(n) \subsetneq \PD_3(n)$ and example of matrices in $\PD_{3}(n)\setminus \PD_{1}(n)$.

First, suppose $q$ is even. Every nonzero matrix in $\PSD_1(n)$ has some principal submatrix which is $\PD_1$ and is therefore $\PD_3$. As a consequence of \cref{lem:posdef 3 q even}, it will have some positive diagonal entry, implying that it belongs to $\PSD_3(n)$. Therefore $\PSD_1(n)\subsetneq \PSD_3(n)$.

Now suppose $q$ is odd. Consider any matrix in $\PSD_1(n)$. If it has full rank, then it is in $\PD_1(n)\subsetneq \PD_3(n)\subsetneq \PSD_3(n)$. On the other hand, if it is not of full rank, then it is in $\PSD_3(n)$ by \cref{thm:psd 3 count}. Therefore, $\PSD_1(n)\subseteq \PSD_3(n)$.

\subsubsection{$\PSD_{4}(n)$ with $\PSD_{1}(n), \PSD_{2}(n)$ and $\PSD_3(n)$} 
The examples in \cref{sec:PD_4_PD_1 2_rel} show that there is no inclusion relation between $\PD_4(n)$ and either $\PD_1(n)$ or $\PD_2(n)$. Consequently, there is no inclusion relation between $\PSD_4(n)$ and either $\PSD_1(n)$ or $\PSD_2(n)$.

It is straightforward to verify that $\PSD_4(n)\subseteq \PSD_3(n)$. Furthermore, the example in \cref{sec:rel posdef} demonstrates that the inclusion is strict.

\subsubsection{$\PSD_{5}(n)$ with $\PSD_{1}(n), \PSD_{2}(n)$ and $\PSD_3(n)$} 
When $q$ is even, \cref{rem:psd 5 even} shows that $\PSD_3(n)\subsetneq \PSD_5(n)$. When $q$ is odd, \cref{lem:posdef 3 q odd} and \cref{thm:psd 3 count} shows that $\PSD_5(n)\subseteq \PSD_3(n)$. 

In \cref{prop:psd 1 psd 5 rel}, we showed that when
$q\equiv3\pmod4$, we have $\PSD_5(n)\subseteq\PSD_1(n)$, and consequently
$\PSD_5(n)\subseteq\PSD_2(n)$. We next show that these inclusions are strict. Furthermore, when $q\equiv1\pmod4$ is odd, we show that there is no
inclusion relation between $\PSD_5(n)$ and $\PSD_i(n)$ for any
$i\in\{1,2\}$.

Choose $a_1\in \mathbb{F}_q^{+}$ such that $a_1+1\in \mathbb{F}_q^{-}$. Let $b_1^2=a_1$. Consider the matrix
\[
B_3=\left(
\begin{array}{c|c}
\begin{matrix}
1 & 0 \\
b_1 & 1 
\end{matrix}
&
\textbf{0} \\
\hline
\textbf{0} & \iden_{n-2}
\end{array}
\right)
\left(
\begin{array}{c|c}
\begin{matrix}
1 & b_1 \\
0 & 1 
\end{matrix}
&
\textbf{0} \\
\hline
\textbf{0} & \iden_{n-2}
\end{array}
\right)=\left(
\begin{array}{c|c}
\begin{matrix}
1 & b_1 \\
b_1 & a_1+1 
\end{matrix}
&
\textbf{0} \\
\hline
\textbf{0} & \iden_{n-2}
\end{array}
\right).
\]
By construction, $B_3\in \PD_2(n)\subseteq \PD_1(n)$ and so, $B_3\in \PSD_1(n)\cap \PSD_2(n)$. Moreover, since $a_1+1\in \mathbb{F}_q^{-}$, we have $B_3 \notin \PSD_5(n)$. Hence, $A_1\in (\PSD_1(n) \cap \PSD_2(n)) \setminus \PSD_5(n)$. 

First, consider $q\equiv 1\pmod{4}$. Let $a_2 \in \mathbb{F}_q^{+}$. Consider the matrix 
\[
B_4=\left(\begin{array}{c|c}
\begin{matrix}
a_2 & -a_2 & 0 \\
-a_2 & a_2 & a_2 \\
0 & a_2 & a_2
\end{matrix}
&
\textbf{0} \\
\hline
\textbf{0} & \iden_{n-3}
\end{array}
\right).
\]
Note that $B_4$ is not a $\PD_1$ matrix (and hence not a $\PD_2$ matrix), as $\det((B_4)_{\{1,2\}})=0$. 
However, $\det B_4 = -a_2^3 \in \mathbb{F}_q^+$,
and hence it is of full rank. 
It is routine to check that all minors of $B_4$ are nonnegative, and hence
$B_4\in\PSD_5(n)$.

\subsubsection{$\PSD_{5}(n)$ and $\PSD_4(n)$} 
Suppose $q$ is even. Since $\PSD_4(n)\subsetneq \PSD_3(n) \subseteq \PSD_5(n)$, we obtain $\PSD_4(n)\subsetneq \PSD_5(n)$.

Now suppose $q$ is odd.
We will now show that there is no inclusion relation between $\PSD_5(n)$ and $\PSD_4(n)$. 
By \cref{thm:x^2+y^2=1}, for $q \geq 7$, we can choose $(a_2,b_2) \in \mathbb{F}_q^2$ such that $a_2^2 + b_2^2 = 1$ and $(a_2,b_2) \notin \{(\pm 1,0), (0,\pm 1)\}$. 
Choose $c_2\in \mathbb{F}_q^{+}$ such that $c_2+1\in\mathbb{F}_q^{-}$. Consider the matrix
\[
\begin{aligned} 
B_5 &=\left(\begin{array}{c|c}
\begin{matrix}
a_2 & b_2  \\
b_2 & -a_2 
\end{matrix} & \textbf{0} \\
\hline
\textbf{0} & \iden_{n-2}
\end{array}\right) 
\left(\begin{array}{c|c}
\begin{matrix}
1 & 0  \\
0 & c_2a_2^2b_2^{-2} 
\end{matrix} & \textbf{0} \\
\hline
\textbf{0} & \iden_{n-2}
\end{array}\right)
\left(\begin{array}{c|c}
\begin{matrix}
a_2 & b_2  \\
b_2 & -a_2 
\end{matrix} & \textbf{0} \\
\hline
\textbf{0} & \iden_{n-2}
\end{array}\right)\\
&=
\left(
\begin{array}{c|c}
\begin{matrix}
a_2^2(c_2+1) & a_2b_2(1-c_2a_2^2b_2^{-2})  \\
a_2b_2(1-c_2a_2^2b_2^{-2}) & c_2a_2^4b_2^{-2}+b_2^2 
\end{matrix}
&
\textbf{0} \\
\hline
\textbf{0} & \iden_{n-2}
\end{array}
\right).
\end{aligned}
\]
Observe that $B_5\in \PSD_4(n)\setminus \PSD_5(n)$, since $a_2^2(c_2+1)$ is negative. 
By \cref{cor:+-}, there exists $a_3\in \mathbb{F}_q$ such that $a_3^2+1\in \mathbb{F}_q^-$. Consider the matrix
\[ 
B_6 =\left(\begin{array}{c|c}
\begin{matrix}
1 & a_3 \\
a_3 & a_3^2 
\end{matrix} & \textbf{0} \\
\hline
\textbf{0} & \iden_{n-2}
\end{array}\right). 
\]
It is easy to check that $B_6 \in \PSD_5(n)$. The characteristic polynomial of $B_6$ is $(x-1)^{n-2}x(x-(1+a_3^2))$. Since $B_6$ always has a negative eigenvalue, it follows that $B_6 \notin \PSD_4(n)$.

\section{Totally positive matrices}
\label{sec:TP}

There is another important family of matrices over the real field, namely totally positive matrices. 

\begin{definition}
\label{def:tp}
A matrix $A\in M_{m,n}(\mathbb{F}_q)$ is said to be \textit{totally positive} if all of its minors are positive; that is $\operatorname{det}(A_{I,J})$ is positive for all $I\subseteq [m]$, $J\subseteq [n]$ of same size. We denote by $\mathrm{TP}_{m,n}(\mathbb{F}_q)$ the set of all totally positive matrices in $M_{m,n}(\mathbb{F}_q)$.
\end{definition}

\subsection{Enumeration}
\label{sec:totpos enum}

In this section, we study totally positive matrices over finite fields. Our goal is to count the number of totally positive matrices in $M_{m,n}(\mathbb{F}_q)$. The next proposition is the first result in this direction.

\begin{proposition}\label{prop:tp count form}
The number of totally positive in $\mathrm{TP}_{m,n}(\mathbb{F}_q)$ is
\[
|\mathrm{TP}_q(m,n)|= r\cdot|\mathbb{F}_q^{+}|^{m+n-1},
\]
where $r$ is the number of totally positive matrices in $\mathrm{TP}_{m,n}(\mathbb{F}_q)$ with all the entries in first row and first column equal to 1. 
\end{proposition}

\begin{proof}
Define a relation $\sim$ on $M_{m,n}(\mathbb{F}_q)$ where $A\sim B$ if $B = PAQ$, where $P \in \mathrm{GL}_m(\mathbb{F}_q)$ and $Q \in \mathrm{GL}_n(\mathbb{F}_q)$ are diagonal matrices with positive diagonal entries. It is trivial to see that $\sim$ is an equivalence relation. We first prove the following two properties of $\sim$.

\medskip

\noindent Claim 1: Equivalence class of a totally positive matrix consists entirely of totally positive matrices.
    
\noindent Proof of the claim: Suppose $A\in \mathrm{TP}_{m,n}(\mathbb{F}_q)$ and $B = PAQ$ for some diagonal matrices $P\in \mathrm{GL}_m(\mathbb{F}_q)$ and $Q\in \mathrm{GL}_n(\mathbb{F}_q)$ with positive diagonal entries. Let 
\[P=\operatorname{diag}(p_1, p_2, \dots, p_m),\quad Q = \operatorname{diag}(q_1, q_2, \dots, q_n).
\]
    For a fix $1\leq k\leq\min\{m,n\}$ and $I = \{i_1, i_2, \dots, i_k\} \subseteq [m]$ and $J = \{j_1, j_2, \dots, j_k\} \subseteq [n]$,
\[ 
\det(B_{I,J}) = \left( \prod_{r=1}^k p_{i_r} q_{j_r} \right) \cdot \det(A_{I,J}). 
\]
    Since all $p_i$'s and $q_j$'s are positive and $A$ is a totally positive matrix, $\det(B_{I,J})$ is also positive. Hence, $B$ is a totally positive matrix. This proves the first claim.

    \medskip
    
    \noindent Claim 2: Each equivalence class of a totally positive matrix contains $\left( |P(\mathbb{F}_q)| \right)^{m + n - 1}$ elements. 
    
    \noindent Proof of the claim: Let $A = (a_{i,j})_{m\times n}\in \mathrm{TP}_{m,n}(\mathbb{F}_q)$ be a totally positive matrix. Define the map $\phi: [A] \rightarrow ({\mathbb{F}_q^{+}})^{m+n-1}$ by
\[ 
\phi(B) = (b_{1,1}, b_{2,1}, \dots, b_{m,1}, b_{1,2}, b_{1,3}, \dots, b_{1,n}), 
\]
    where $B = (b_{i,j})_{m\times n} \in [A]$.

    We first show that $\phi$ is surjective. Let
\[ 
b = (b_{1,1}, b_{2,1}, \dots, b_{m,1}, b_{1,2}, b_{1,3}, \dots, b_{1,n}) \in ({\mathbb{F}_q^{+}})^{m+n-1}.
\]
    Define
\[ 
P = \operatorname{diag}\left( \frac{b_{1,1}}{a_{1,1}}, \frac{b_{2,1}}{a_{2,1}}, \dots, \frac{b_{m,1}}{a_{m,1}} \right), 
\]
and
\[ 
Q = \operatorname{diag}\left( 1, \frac{a_{1,1} b_{1,2}}{b_{1,1} a_{1,2}}, \dots, \frac{a_{1,1} b_{1,n}}{b_{1,1} a_{1,n}} \right). 
\]
    Then $PAQ \in [A]$ and satisfies $\phi(PAQ) = b$. Hence, $\phi$ is surjective.

    Now, we show $\phi$ is injective. Suppose $B=(b_{i,j})_{m\times n}, C=(c_{i,j})_{m\times n}\in [A]$ such that $\phi(B)=\phi(C)$. Let $P=\operatorname{diag}(p_1, \dots, p_m)$ and $Q = \operatorname{diag}(q_1, \dots, q_n)$ with positive diagonal entries such that $C=PBQ$. Therefore
\[ p_i q_1 b_{i,1} = c_{i,1} = b_{i,1} \quad \text{and} \quad p_1 q_j c_{1,j} = b_{1,j} = c_{1,j},
\]
    and since $c_{i,j} \neq 0$, we deduce that $p_i q_1 = 1$ for all $i\in [m]$ and $p_1 q_j = 1$ for all $j\in [n]$. So $p_i=p_1$ and $q_j = q_1$ for all $i, j$, implying $PBQ = B$, that is, $B = C$. Thus, $\phi$ is injective. This proves the second claim.

    Therefore, by Claims 1 and 2, the relation $\sim$ partitions the set $\mathrm{TP}_{m,n}(\mathbb{F}_q)$ into equivalence classes, each of size $|\mathbb{F}_q^{+}|^{m+n-1}$. Moreover, by Claim 2, each equivalence class contains a unique totally positive matrix with all entries in the first row and first column entries equal to $1$. Hence, the total number of totally positive matrices in $\mathrm{TP}_{m,n}(\mathbb{F}_q)$ is $r \cdot |\mathbb{F}_q^{+}|^{m+n-1}$, which proves the proposition.
\end{proof}

An observation from the proof of \cref{prop:tp count form} is that every totally positive matrix is equivalent to a unique totally positive matrix whose entries in the first row and the first column are all equal to 1. Since matrices of this form will play a significant role in what follows, we introduce the following terminology.

\begin{definition}
We called a totally positive matrix whose entries in the first row and the first column are all equal to 1 an \emph{elementary totally positive matrix}.
\end{definition}

We now gives a formula for the number of elementary totally positive matrices for small orders.

\begin{lemma}\label{lem:elem tpm order 2 3}
The number of elementary totally positive matrices in $\mathrm{TP}_{2,2}(\mathbb{F}_q)$ is,
\[
\begin{cases}
q-2 & \text{if $q$ is even},\\[0.5em]
\dfrac{q-5}{4} & \text{if $q \equiv 1 \pmod{4}$},\\[0.5em]
\dfrac{q-3}{4} & \text{if $q \equiv 3 \pmod{4}$}.
\end{cases}
\]
If $\mathbb{F}_q$ has characteristic $2$, then the number of elementary totally positive matrices in $\mathrm{TP}_{3,3}(\mathbb{F}_q)$ is
\[(q-2)(q-3)(q^2-9q+21).
\]
\end{lemma}

\begin{proof}
We need to determine the number of ways in which the $(2,2)$-entry of the elementary totally positive matrix can be chosen. Observe that $a \in \mathbb{F}_q$ can be the $(2,2)$-entry of elementary totally positive matrix if and only if $a \in \mathbb{F}_q^{+}$ and $a-1 \in \mathbb{F}_q^{+}$. Therefore, by \cref{thm:a a-1 is +ve}, we obtain the first part.

We now find the number of elementary totally positive matrices in $M_3(\mathbb{F}_q)$. Let
\[A=\begin{pmatrix}
        1 & 1 & 1\\
        1 & a & b\\
        1 & c & d
    \end{pmatrix}
\]
    where $a,b \in \mathbb{F}_q \setminus \{0,1\}$ and $a \neq b$ are fixed. There are $q^2$ choices of $c,d$; however, they must satisfy the following conditions so that $A$ is totally positive:
    \begin{enumerate}
        \item $c$ or $d$ must not be zero: This removes $2q-1$ choices of $c,d$.
        \item $c$ or $d$ cannot be equal to $1$: This removes an additional $2q-3$ choices of $c,d$.
        \item $c \neq d$: This removes an additional $q-2$ choices of $c,d$.
        \item $c \neq a$ or $d \neq b$: This removes an additional $2q-7$ choices of $c,d$.
        \item $ad \neq bc$: This removes an additional $q-4$ choices of $c,d$.
        \item The last row of $A$ is not linearly dependent on the first two rows of $A$: This removes an additional $q-4$ choices of $c,d$.
    \end{enumerate}
    Therefore, there are $(q^2-9q+21)$ choices for $c,d$.

    Since there are $(q-2)(q-3)$ choices of $a,b$, the number of elementary totally positive matrices in $\mathrm{TP}_{3,3}(\mathbb{F}_q)$ is
\[(q-2)(q-3)(q^2-9q+21).
\]
    This proves the second part.
\end{proof}

Now, we derive formulas for the number of totally positive matrices of small order.
\begin{itemize}
\item By \cref{prop:tp count form} and \cref{lem:elem tpm order 2 3}, the number of totally positive matrices in $M_{2}(\mathbb{F}_q)$ is
\[
|\mathrm{TP}_{2,2}(\mathbb{F}_q)| =
\begin{cases}
(q - 2)(q - 1)^3 & \text{if } q \text{ is even}, \\
\left(\dfrac{q - 5}{4}\right)\left(\dfrac{q - 1}{2}\right)^3 & \text{if } q \equiv 1 \pmod{4}, \\
\left(\dfrac{q - 3}{4}\right)\left(\dfrac{q - 1}{2}\right)^3 & \text{if } q \equiv 3 \pmod{4}, 
\end{cases}
\]
and the number of totally positive matrices in $M_{3}(\mathbb{F}_q)$ when $q$ is even is
\[
|\mathrm{TP}_{3,3}(\mathbb{F}_q)|=(q-1)^5(q-2)(q-3)(q^2-9q+21)
\]

\item When $q$ is even, the total number of totally positive matrices in $M_{2,k}(\mathbb{F}_q)$ is
\[
|\mathrm{TP}_{2,k}(\mathbb{F}_q)|=\left(\prod_{i=2}^k(q-i)\right)(q-1)^{k+1}=(k-1)! \binom{q - 2}{k - 1}(q - 1)^{k + 1},
\]
as there are $q - 2$ choices for $a_{2,2}$, $q - 3$ choices for $a_{2,3}$, $q - 4$ choices for $a_{2,4}$, and so on.

\item Using \texttt{SageMath}~\cite{sagemath}, we find that $|\mathrm{TP}_{3,3}(\mathbb{F}_{q})|=0$ for odd $q$ and $q\leq 17$. 
For the next primes, we obtain
\begin{enumerate}
    \item $|\mathrm{TP}_{3,3}(\mathbb{F}_{19})|=9,904,396=2^2\cdot 19^5$
    \item $|\mathrm{TP}_{3,3}(\mathbb{F}_{23})|=38,618,058=2\cdot 3\cdot23^5$
    \item $|\mathrm{TP}_{3,3}(\mathbb{F}_{25})|=58,593,750=2\cdot 3\cdot 5^{10}$
\end{enumerate}

\item For $3 \times 4$ matrices for $q$ even, we find again using \texttt{SageMath}: 
\begin{enumerate}
    \item $|\mathrm{TP}_{3,4}(\mathbb{F}_{8})|=1,200=2^4\cdot 3\cdot 5^2$
    \item $|\mathrm{TP}_{3,4}(\mathbb{F}_{16})|=2,075,472=2^4\cdot 3^2\cdot 7\cdot 29\cdot 71$
    \item $|\mathrm{TP}_{3,4}(\mathbb{F}_{32})|=415,217,040=2^4\cdot 3\cdot5\cdot 7\cdot 89\cdot 2777$
\end{enumerate}

\end{itemize}

\subsection{Bounds relating the order and $q$}

For a fixed $q$, we now establish an upper bound on the order for which a totally positive matrix over $\mathbb{F}_q$ can exist. We first show that when $q\equiv 3 \pmod{4}$, interchanging rows or columns of a totally positive matrix may result in a matrix that is no longer totally positive. We begin by recalling the definition of a permutation matrix. For a permutation $\pi:[n]\to[n]$, the corresponding \textit{permutation matrix} is the matrix $P_\pi=(p_{i,j})_{n\times n}$, where $p_{i,j}=1$ if $\pi(i)=j$ and $p_{i,j}=0$ otherwise.

\begin{lemma}\label{lem: perm action on psd 5 q cong 3}
    Let $q \equiv 3 \pmod{4}$ be a prime power, and let $A \in \mathrm{TP}_{m,n}(\mathbb{F}_q)$. Then for any nonidentity permutation matrices $P \in \mathrm{GL}_m(\mathbb{F}_q)$ and $Q \in \mathrm{GL}_n(\mathbb{F}_q)$, the matrices $PA$ and $AQ$ are not totally positive. 
\end{lemma}

\begin{proof}
  Let $P$ correspond to a permutation $\pi \in S_n$. Since $\pi$ is not the identity permutation, there exists indices $i, j \in [n]$ such that $i<j$ and $\pi(j)<\pi(i)$. The determinant of the submatrix $PA_{\{i,j\}, \{1,2\}}$ is negative, so $PA$ is not a totally positive matrix. A similar argument shows that $AQ$ is also not totally positive matrix.
\end{proof}

\begin{proposition}\label{prop:bound tpm}
    If there exists a totally positive matrix in $M_{m,n}(\mathbb{F}_q)$ where $m, n \geq 2$, then
\[
q \geq 
\begin{cases}
    \max(m+1, n+1) & \text{ for even } q \\
    \max(4m+1, 4n+1) & \text{ for } q \equiv 1 \pmod{4}\\
    4(m+n) - 9 & \text{ for } q \equiv 3 \pmod{4}.
\end{cases}
\]
\end{proposition}
    
\begin{proof}
Since every totally positive matrix is equivalent to an elementary totally positive matrix of the same order and the inequality only involve the sizes of he matrices and the field, we prove the proposition for elementary totally positive matrices.

First, consider the case when $q \not\equiv 3 \pmod{4}$. For any fixed row or column, no two entries are equal. Moreover, by \cref{lem:elem tpm order 2 3}, at most $q-2$ entries can appear in a fixed row or column when $q$ is even, while at most $(q-5)/2$ entries can appear in a fixed row or column when $q \equiv 1 \pmod{4}$. This establishes the proposition in this case.

Now consider the case when $q \equiv 3 \pmod{4}$. Let $X=\{x_1,x_2,\dots,x_k\}$ be a ordered subset of positive elements such that $x_i-1$ is positive for all $i$, and $x_j-x_i$ is positive whenever $j>i$. Suppose $A=(a_{i,j})_{m \times n}$ is an elementary totally positive matrix. For $i \in [2, m]$ and $j \in [2, m]$, the submatrix $A_{[\{1,i\},\{1,j\}]}$ is totally positive; hence $a_{i,j} \in X$.

Define a function $f :[2, m]^2 \to [k]$ by $f(i,j) = \ell$ whenever $a_{i,j} = x_\ell$.
By \cref{lem: perm action on psd 5 q cong 3}, the positivity of $\det(A_{\{1,i\},\{j_1,j_2\}})$ implies that, for each fixed $r\in [2, m]$, we have
\[
f(r,2) < f(r,3) < \cdots < f(r,n).
\]
Similarly, the positivity of $\det(A_{\{i_1,i_2\},\{1,j\}})$ implies that, for each fixed $s\in [2, n]$, we have
\[
f(2,s) < f(3,s) < \cdots < f(m,s).
\]
By iteratively applying these inequalities, together with the fact that $f(2,2) \geq 1$, we deduce that
\[
f(i,j) \geq i + j - 3 \qquad \text{for all } i\in [2, m]\text{ and } j\in [2, n].
\]
In particular, this implies
\[
m + n \leq f(m,n) + 3 \leq k + 3.
\]
By \cref{lem:elem tpm order 2 3}, we have $k \leq (q-3)/4$. Therefore,
\[
m + n \leq k + 3 \leq \frac{q+9}{4},
\]
which completes the proof.
\end{proof}

Observe that in the proof of \cref{prop:bound tpm}, only minors of order $2$ of the form $A_{[\{1,i\},\{1,j\}]}$ are considered. In particular, no other  minors of order $2$ are considered, and minors of order greater than $2$ are not used in the proof. Consequently, the bound obtained in \cref{prop:bound tpm} is likely far from optimal and can be improved significantly. This motivates the following question.

\begin{question}
For fixed positive integers $m$ and $n$, find the minimal prime power $q$ such that there exists an $m \times n$ totally positive matrix over $\mathbb{F}_q$.
In particular how does this minimal value of $q$ grow with $m$ and $n$?
\end{question}

\subsection{Structural formula}

At present, we do not have a simple closed-form expression for the number of totally positive matrices in $M_{m,n}(\mathbb{F}_q)$ when $\mathbb{F}_q$ is of odd characteristic and $\max\{m,n\} \geq 3$, or when $\mathbb{F}_q$ is of characteristic $2$ and $m,n \geq 3$. In the absence of such a formula, we instead derive a structural formula for the number of totally positive matrices in $M_{m,n}(\mathbb{F}_q)$ using an approach similar to that of \cref{thm:psd 5 struc formula}.

Observe that the number of totally positive matrices in $\mathrm{TP}_{m,n}(\mathbb{F}_q)$ and $\mathrm{TP}_{n,m}(\mathbb{F}_q)$ is the same, so in the rest of this section we can assume that $m \leq n$. 

First consider the case when $q$ is even. For each $r \leq m$ and $q=2^k$, let $\mathcal{X}_{r,k} $ denote the set of all matrices $A\in M_{m,n}(\mathbb{F}_{2^k})$ whose every $r \times r$ submatrix has positive determinant. In particular,
\begin{equation}\label{eq:tot_pos as intersection}
    \mathrm{TP}_{m,n}(\mathbb{F}_{2^k})=\bigcap_{r=1}^m \mathcal{X}_{r,k}
\end{equation}
For every $r\leq m$, let $\mathcal{A}_{r,k}$ be set of all matrices in $ M_{m,n}(\mathbb{F}_{2^k})$, which has at least one singular $r\times r$ submatrix. Therefore $\mathcal{A}_{r,k}=\mathcal{X}_{r,k}^c$ for all $r\leq m$ and positive integer $k$. From \eqref{eq:tot_pos as intersection}, we have 
\[
\mathrm{TP}_{m,n}(\mathbb{F}_{2^k})=\bigcap_{r=1}^m \mathcal{A}_{r,k}^c=\left(\bigcup_{r=1}^m\mathcal{A}_{r,k}\right)^c.
\]
Therefore, taking cardinality, we get 
\begin{equation}\label{eq:tot_pos as intersection even count}
    |\mathrm{TP}_{m,n}(\mathbb{F}_{2^k})|=2^{kmn}-\left|\bigcup_{r=1}^m\mathcal{A}_{r,k}\right|
\end{equation}

We are now ready to state the structural formula for the number of totally positive matrices in $M_{m,n}(\mathbb{F}_{2^k})$.

\begin{theorem}
    For every positive integers $m,n$, there exists positive integer $t$, integers $c_{1},c_{2},\dots,c_{t}$ and algebraic numbers $\gamma_{1},\gamma_{2},\dots,\gamma_{t}$ such that the number of totally positive matrices in $M_{m,n}(\mathbb{F}_{2^k})$ is of the form
\[c_{1} \gamma_{1}^{k} + c_{2} \gamma_{2}^{k} + \cdots + c_{t} \gamma_{t}^{k}.
\]
\end{theorem}

\begin{proof}
    By the inclusion exclusion principle, we have
\begin{equation}\label{eq2}
    \left| \bigcup_{i=1}^m \mathcal{A}_{i,k} \right|
= \sum_{j=1}^m (-1)^{j-1} 
\left(
\sum_{1 \leq i_1 < i_2 < \cdots < i_j \leq m }
\left| \mathcal{A}_{i_1,k} \cap \cdots \cap \mathcal{A}_{i_j,k} \right|
\right).
\end{equation}
Fix $1 \leq i_1 < i_2 < \cdots < i_{\ell} \leq m$, and let $X = (x_{a,b})_{1 \leq a\leq m,1\leq b \leq n}$ be an $m \times n$ matrix of variables. The set of matrices $\mathcal{A}_{i_1,k} \cap \cdots \cap \mathcal{A}_{i_{\ell},k}$ can be viewed as an algebraic set over the affine space $M_{m,n}(\mathbb{F}_{2^k})$, where the defining polynomials of this algebraic set in $\mathbb{F}_2[X]$ are
\[
f_j(x_{a,b} \mid 1 \leq a,b \leq n)=\prod_{I\in \binom{[m]}{i_j},J\in \binom{[n]}{i_j}}\det\left(X_{I,J}\right)=0
\]
where $j\in [{\ell}]$.
Applying \cref{lem:rational funcn} on $f_1,f_2,\dots,f_{\ell}$, we obtain algebraic numbers $\alpha_1,\alpha_2,\dots,\alpha_{n_1},\beta_1,\beta_2,\dots,\beta_{n_2}$, such that
\[
\left|\mathcal{A}_{i_1,k} \cap \mathcal{A}_{i_2,k} \cap \cdots \cap \mathcal{A}_{i_{\ell},k}\right|=\sum_{j=1}^{n_2}\beta_j^k -\sum_{i=1}^{n_1}\alpha_i^k
\]
Putting this in \eqref{eq2}, gives 
\[ 
\left| \bigcup_{i=1}^m \mathcal{A}_{i,k} \right|=\sum_{j=1}^{t}c_j\gamma_j'^k
\]
where $c_j$ is integer and $\gamma_j'$ are algebraic numbers.  Therefore, by \eqref{eq:tot_pos as intersection even count}, 
\[|\mathrm{TP}_{m,n}(\mathbb{F}_q)|=2^{kmn}-\left|\bigcup_{r=1}^mA_{r,k}\right|=2^{kmn}-\sum_{j=1}^{t}c_j\gamma_j'^k
\]
Which is of the desire form.
\end{proof}

Now suppose $p$ is odd. Let $\mathcal{I}$ be the set of all ordered pairs $(I,J)$ of nonempty subsets $I\subseteq [m]$ and $J\subseteq [n]$ with $|I|=|J|$. Let $X = (x_{a,b})_{m\times n}$ be an $m \times n$ matrix of variables, and let $Y=\{y_{I,J} \mid (I, J) \in\mathcal{I}\}$ be set of additional variables.

Consider the following system of polynomial equations in $\mathbb{F}_p[X,Y]$ defined by:
\[
\{g_{I,J}(X,Y)\mid (I,J)\in\mathcal{I}\}
\]
where 
\[
g_{I,J}(X,Y)=y_{I,J}^2 \det\left(X[I,J]\right) - 1.
\]
for $(I,J)\in\mathcal{I}$, and for each positive integer $k$, let $\mathcal{X}_{k}^0 \subseteq \mathbb{F}_{p^k}^{m\cdot n+|\mathcal{I}|}$ be the algebraic set defined by the above system of polynomial equations.

The following lemmas can be proved using arguments analogous to those in \cref{lem:part psd 5 1,lem:part psd 5 2}; therefore, we omit the proofs.

\begin{lemma}\label{lem:part tot pos 5 1}
If $(X',Y') \in \mathcal{X}_k^0$, then $X' \in \mathrm{TP}_{m,n}(\mathbb{F}_q)$. Conversely, if $X' \in \mathrm{TP}_{m,n}(\mathbb{F}_q)$, then there exists $Y'=(y_{I,J})_{(I,J) \in \mathcal{I}} \in \mathbb{F}_{p^m}^{|\mathcal{I}|}$ such that $(X',Y') \in \mathcal{X}_k^0$.
\end{lemma}

\begin{lemma} \label{lem:part tot pos 5 2}
If $(X',Y') \in \mathcal{X}_k^0(\epsilon)$, then $(X', \sigma Y') \in \mathcal{X}_k^0$ for every $\sigma \in \{-1,1\}^{t}$. Moreover, if $(X',Y'), (X',Y_0') \in \mathcal{X}_k^0$, then there exists $\sigma \in \{-1,1\}^{t}$ such that $Y' = \sigma Y_0'$.
\end{lemma}

We are now ready to state the structural formula for the number of totally positive matrices in $M_{m,n}(\mathbb{F}_{p^k})$ when $p$ is odd. 

\begin{theorem}\label{thm:tp struc formula}
For odd prime $p$ and positive integers $m,n$, there exists a positive integer $t$, integers $c_{1},c_{2},\dots,c_{t}$, and algebraic numbers $\gamma_{1},\gamma_{2},\dots,\gamma_{t}$ such that the number of totally positive matrices in $M_{m,n}(\mathbb{F}_q)$, where $q=p^k$, is of the form
\[
2^{-|\mathcal{I}|}\left(c_{1}\gamma_{1}^{k}+c_{2}\gamma_{2}^{k}+\cdots+c_{t}\gamma_{t}^{k}\right).
\]
\end{theorem}

\begin{proof} 
We can define a equivalence relation $\sim$ on the set $\mathcal{X}_k^0$ where we define $(X',Y') \sim (X_0',Y_0')$ if $X' = X_0'$.  By \cref{lem:part psd 5 1}, each equivalence class corresponds to a unique element of $\mathrm{TP}_{m,n}(\mathbb{F}_{p^k})$, and by \cref{lem:part psd 5 2}, each equivalence class has cardinality $2^{|\mathcal{I}|}$. Hence,
\begin{equation}\label{eq:rel tp count to X0_k}
|\mathcal{X}_k^0| = 2^{|\mathcal{I}|}\cdot \left|\mathrm{TP}_{m,n}(\mathbb{F}_{p^k})\right|.
\end{equation} 
Using \cref{lem:rational funcn} on 
\[
\{g_{I,J}\mid (I,J)\in\mathcal{I}\},
\]
we obtain algebraic numbers $\alpha_1,\alpha_2,\dots,\alpha_{n_1},\beta_1,\beta_2,\dots,\beta_{n_2}$ such that 
\[
\left|\mathcal{X}_k^0\right|=\sum_{j=1}^{n_2}\beta_j^k -\sum_{i=1}^{n_1}\alpha_i^k
\]
Plugging this into \eqref{eq:rel tp count to X0_k}, we will get the desire form.
\end{proof}

By reasoning similar to that in \cref{rem:when psd 5 polyn} and the formulas for the number of totally positive matrices for small orders given in \cref{sec:totpos enum}, we make the following conjecture.

\begin{conjecture}
For a prime $p$ and positive integers $m, n$, there exists a positive integer $t$ and rational polynomials $g_0(x),\dots,g_{t-1}(x)\in \mathbb{Q}[x]$ of degree $m n$, such that the number of totally positive matrices in $M_{m,n}(\mathbb{F}_{p^k})$ is $g_i(p^k)$ when $k\equiv i\pmod{t}$.
\end{conjecture}

\section*{Acknowledgements}
We thank J. Cooper, S. Damase, D. Guillot and A. Khare for discussions.
We acknowledge support from the DST FIST Program 2021 TPN–700661. 
The second named author (SP) wishes to thank the Council of Scientific and Industrial Research (CSIR), New Delhi, for the award of a Junior Research Fellowship [File No. 09/0079(19932)/2024-EMR-I].

\bibliographystyle{alpha}
\bibliography{main}

\appendix

\section{A new characterization for positive semidefinite matrices}
\label{sec:new psd}
We now prove the equivalence of \cref{thm:psd RC}(a) and (f). We begin by recalling the necessary lemmas.

\begin{lemma}[{\cite[Theorem 1 and Theorem 3 of Section 20]{Bocher1922}}]\label{lem:bocher thm}
Let $\mathbb{F}$ be a field, and let $A\in M_n(\mathbb{F})$ be a symmetric
matrix. Suppose there exists a subset $K\subseteq[n]$ of size $r$ such that
$A_K$ is nonsingular. If, for every subset $M$ satisfying
$K\subseteq M\subseteq[n]$ and $|M|\in\{r+1,r+2\}$, the principal submatrix
$A_M$ is singular, then $A$ has rank $r$. Conversely, if rank of a symmetric matrix is $r>0$, then there exist atleast one $K\subseteq[n]$ of size $r$ such that $A_K$ is nonsingular.
\end{lemma}

\begin{lemma}[Cauchy interlacing theorem {\cite[Theorem 4.3.17]{HornJohnson2012}}]\label{thm:cauchy int thm} Let $B\in M_{n}(\mathbb{C})$ be a Hermitian matrix partitioned as
\[
B=
\left(
\begin{array}{c|c}
A & v\\
\hline
\\[-1em]
v^T &a_2
\end{array}
\right),
\]
where  
$A \in M_{n-1}(\mathbb{C}), v \in \mathbb{C}^{n-1}$ and $a_2\in\mathbb{C}$. Then
\[
\lambda_1(B)\leq \lambda_1(A)\leq \lambda_2(B)\leq\dots\leq\lambda_{n-1}(B)\leq \lambda_{n-1}(A)\leq \lambda_{n}(B).
\]
\end{lemma}

\begin{proof}[Proof of equivalence of \cref{thm:psd RC}(a) and (f)] 
Suppose $A$ is a psd matrix of rank $r$. If $r=0$, there is nothing to prove. Suppose $r>0$. By \cref{lem:bocher thm}, there exists a subset $K\subseteq[n]$ of size $r$ such that $A_K$ is nonsingular. Since $A$ is positive semidefinite of rank $r$, it has exactly $r$ positive eigenvalues, while $0$ is an eigenvalue of multiplicity $n-r$. By \cref{thm:cauchy int thm}, every eigenvalue of $A_K$ is nonnegative. As $A_K$ is nonsingular, none of its eigenvalues is zero. Hence all eigenvalues of $A_K$ are positive, and therefore $A_K$ is positive definite.

Now let $M\subseteq[n]$ satisfy $|M|\in \{r+1,r+2\}$. Since $\operatorname{rank}(A)=r$, every principal submatrix of $A$ of order greater than $r$ is singular. In particular, $A_M$ is singular.

Conversely, suppose that $A$ satisfies \cref{thm:psd RC}(f). By \cref{lem:bocher thm}, $\operatorname{rank}(A) = r$. Let $0 < \lambda_1 \le \dots \le \lambda_r$ be the eigenvalues of $A_K$. By recursively applying \cref{thm:cauchy int thm}, $A$ has at least $r$ positive eigenvalues. Since $\operatorname{rank}(A)=r$, the eigenvalue $0$ has multiplicity $n-r$. Hence, all eigenvalues of $A$ are nonnegative, and therefore $A$ is positive semidefinite of rank $r$.
\end{proof}

\begin{proof}[Proof of \cref{thm:psd princ submatrix}]
We will show that $A$ satisfies \cref{thm:psd RC}(f) if and only if there exists a subset $K\subseteq[n]$ such that $A_K$ is positive definite and every principal submatrix $A_M$ is singular for all subsets $M\subseteq[n]$ satisfying $K\subseteq M$ and $|M|>|K|$.

If $A$ satisfies \cref{thm:psd RC}(f), then $A_K$ is positive definite. By \cref{lem:bocher thm}, the rank of $A$ is $|K|$. Hence every principal submatrix $A_M$ with $|M|>|K|$ is singular. The converse is trivial.
\end{proof}

\end{document}